\newcommand\dfn[1]{{\bf #1}}  
\documentclass{amsart}

\usepackage{amsthm}
\usepackage{amssymb}
\usepackage[all]{xy}   
\usepackage[T1]{fontenc}

\input xy
\xyoption{all} 
\providecommand{\cal}{\mathcal}
\newcommand{\calA}{{\cal{A}}}

\newcommand{\calU}{{\cal{U}}}

\newcommand{\calTT}{{\hbox{\scriptsize $\mathcal{T}$}}}

\newcommand{\dpower}[2]{[#1]^{#2}}
\newcommand{\Nat}{{\mathbb{N}}}
\newcommand{\ntr}{{n\in\omega}}
\newcommand{\El}{{\cal{L}}}

\newcommand{\etaf}{\hat{\eta}}

\newcommand{\la}{\langle}
\newcommand{\ra}{\rangle}
\newcommand{\Ra}{\Rightarrow}

\newcommand{\dd}{{\downarrow}}
\newcommand{\uu}{{\uparrow}}

\newcommand{\IR}{\mathbb R}
\newcommand{\natN}{\mathbb N}
\newcommand{\natQ}{\mathbb Q}

\newcommand{\Clop}{\mathrm{Clop}}

\newcommand{\Endpt}{\mathrm{Endpt}}
\newcommand{\lastpt}{\mathrm{endpt}}

\newcommand{\tip}{\mathrm{tip}}

\newcommand{\E}{\mathscr E}

\newcommand{\Fin}[1]{[#1]^{<\omega}} 
\newcommand{\Finnn}[1]{[#1]_*^{<\omega}}  

\newcommand{\w}{\omega}

\newcommand{\Ord}{\mathrm{Ord}}

\newcommand{\cl}{\mathrm{cl}}

\newcommand{\rank}{{\mathrm{rk}\scriptstyle{W\negmedspace F}}}
\newcommand{\height}{{\mathrm{ht}\scriptstyle{C\negthinspace B}}}

\theoremstyle{plain}

\newtheorem{theorem}{Theorem}[section]
\newtheorem{theorema}{Theorem}  
\newtheorem{corollary}[theorem]{Corollary}
\newtheorem{proposition}[theorem]{Proposition}
\newtheorem{lemma}[theorem]{Lemma}

\newtheorem{fact}[theorem]{Fact}
\newtheorem{algorithm}[theorem]{Computation Rules}
\newtheorem{Fact}{\rm Fact}
\newtheorem{Example}[Fact]{\rm Example}

\theoremstyle{definition}

\newtheorem{question}[theorem]{Question}
\newtheorem{example}[theorem]{Example}
\theoremstyle{remark}
\newtheorem{remark}[theorem]{Remark}

\newtheorem*{theorem*}{Theorem} 
\newtheorem*{theoremA*}{Theorem A} 
\newtheorem*{corollary*}{Corollary}
\newtheorem*{proposition*}{Proposition}
\newtheorem*{lemma*}{Lemma}
\newtheorem*{claim*}{Claim}
\newtheorem*{fact*}{Fact}
\newtheorem*{observation*}{Observation}
\newtheorem*{remark*}{Remark}
\newtheorem*{example*}{Example}
\newtheorem*{comment*}{Comment}
\newtheorem*{question*}{Question}
\newtheorem*{construction*}{Construction}

\usepackage{mathrsfs}
\newcommand{\powerset}{\mathscr{P}}
\newcommand{\mA}{\mathscr{A}}
\newcommand{\mB}{\mathscr{B}}

\newcommand{\mF}{\mathscr{F}}

\newcommand{\mL}{\mathscr{L}}

\newcommand{\mS}{\mathscr{S}}
\newcommand{\mU}{\mathscr{U}}
\newcommand{\mV}{\mathscr{V}}
\newcommand{\mW}{\mathscr{W}}

\newcommand{\join}{\vee}
\renewcommand{\Join}{\bigvee}

\newcommand{\Min}{\mathrm{Min}}
\newcommand{\Max}{\mathrm{Max}}

\newcommand{\wh}[1]{\widehat{#1}}

\newcommand{\eqdf}{\hbox{\bf \,:=\,}}

\newcommand{\img}[2]{#1[#2]} 
\newcommand{\inv}[2]{{#1}^{-1}[#2]} 

\newcommand{\pair}[2]{\langle #1, #2 \rangle} 

\newcommand{\setof}[2]{\{#1\colon #2\}}
\newcommand{\SETOF}[2]{\bigl\{#1\colon #2\bigr\}}

\newcommand{\seqof}[2]{\langle #1\colon #2\rangle}
\newcommand{\seqn}[2]{\langle  #1\rangle_{#2}}
\newcommand{\seqnn}[2]{\langle  #1\rangle_{#2}}
\newcommand{\sett}[2]{\{#1\}_{#2}}

\newcommand{\fnn}[3]{#1:#2 \rightarrow #3} 

\newcommand{\Ult}{\mathrm{Ult}}
\newcommand{\FS}{\mathrm{FS}}
\newcommand{\IS}{\mathrm{IS}}

\makeindex

\usepackage[colorlinks=true,urlcolor=blue, citecolor=blue,
linkcolor=blue, bookmarks=true]{hyperref}

\newcommand{\BULLET}{\mbox{\tiny$\bullet$}}

\begin{document} 

\title[Vietoris and Priestley spaces]{Vietoris hyperspaces over scattered Priestley spaces}

	\author[Banakh]{Taras Banakh}
	\address{Ivan Franko National University 
	of Lviv (Ukraine) and
	Jan Kochanowski University, Kielce (Poland)}
	\email{t.o.banakh@gmail.com} 

	\author[Bonnet]{Robert Bonnet}
	\address{Laboratoire de Math\'ematiques (LAMA), 
	Universit\'e de Savoie Mont Blanc, 	(France)}
	\email{bonnet@univ-smb.fr \  and \   			   			Robert.Bonnet@math.cnrs.fr}

	\author[Kubi\'s]{Wies\l aw Kubi\'s}
	\address{Institute of Mathematics,
	Czech Academy of Sciences, Praha (Czech Republic)} 
	\email{kubis@math.cas.cz}
	\thanks{The first author has been partially supported by NCN grant
	DEC-2012/07/D/ST1/02087 (The National Science Center, Poland).
	Research of the second author was supported 
	by the Institute of Mathematics of the Czech Academy of Sciences.
	The third author has been supported by the GA \v CR grant 20-22230L (Czech Science Foundation).}
\date{\today}


\begin{abstract}
	We study Vietoris hyperspaces of closed and final sets of Priestley spaces. We are particularly interested in Skula topologies.
	A topological space is {\em Skula} if its topology is generated by differences of open sets of another topology. A compact Skula space is scattered and moreover has a natural well-founded ordering compatible with the topology, namely, it is a Priestley space. One of our main objectives is investigating Vietoris hyperspaces of general Priestley spaces, addressing the question when their topologies are Skula and computing the associated ordinal ranks. We apply our results to scattered compact spaces based on certain almost disjoint families, in particular, Lusin families and ladder systems.
\end{abstract}

\centerline{\small \bf 
Dedicated to the memory of Matatyahu Rubin (1946-2017)}

\maketitle

\tableofcontents

\section{Introduction}
\label{section-1}

This paper is motivated by a problem of detecting scattered
compact spaces that are homeomorphic to (or embed into)
scattered compact topological semilattices.
This problem has been addressed by T. Banakh, O. Gutik and M. Rajagopalan in \cite{GRS, BGR}. 
The papers of R.~Bonnet and M.~Rubin~\cite{BR1}, 
and of A.~Dow and S.~Watson~\cite{DW},
develop classes of compact scattered spaces with a closed partial ordering.
Such orders are ``well-founded''.
We study the relationship between the Cantor-Bendixson height
and the well-founded rank in these spaces.

A topological space is called a \dfn{Skula space} if its topology is generated by differences of open sets of another topology on the same set. This concept was invented by Ladislav Skula~\cite[Result 2.2]{Skula} in 1969, answering a purely category-theoretic question. Much later, in 1990, Dow and Watson~\cite{DW} observed that compact Skula spaces are scattered and they related this concept to well-generated Boolean algebras introduced by Bonnet and Rubin~\cite{BR1}.

A \dfn{Priestley space} is a compact 0-dimensional space endowed with a closed partial order such that clopen final subsets separate points (a set $S$ is \dfn{final} if $x \in S$, $x < y$ imply $y \in S$).
Priestley duality~\cite{{P}} establishes the correspondence between Priestley spaces and bounded distributive lattices, exactly in the same manner as Stone duality between compact 0-dimensional spaces and Boolean algebras. In fact, every compact 0-dimensional space is Priestley when endowed with the trivial ordering, therefore Priestley duality ``contains'' Stone duality. 

As it happens, a compact Skula space has a natural closed partial ordering that is well-founded and makes it a Priestley space. 
Our aim is to give more insight into this relation. 
In particular, we look at Vietoris hyperspaces of Priestley 
spaces, investigating when their topology is Skula.
In that case, scatteredness and 
well-foundedness are intimately linked, 
and we study for a Skula space space and its 
Vietoris hyperspace the relationship 
between its Cantor-Bendixson height and 
its well-founded rank.
We also consider a specialization: canonically Skula spaces. As we have mentioned above, one of our motivations is topological semilattice theory in which Priestley spaces play a significant role: The free semilattice over a Priestley space is its Vietoris hyperspace (see below for more details).

\subsection*{Preliminaries}

A \dfn{(join) semilattice}, is a set $X$ endowed with a binary operation, $\join : X \times X \to X$: the operation $\join$ should be associative, commutative, and idempotent (in the sense that $x \join x=x$ for all $x\in X$).
Every semilattice has a natural partial ordering: $x \leq y$ if and only if  $x \join y = y$.
By a \dfn{topological semilattice}
we mean a Hausdorff topological space $X$
endowed with a continuous join semilattice operation
$\join$.
In that case $\leq$ is closed (as subset of $X {\times}X$).

A \dfn{poset} is a partially ordered set.
For a point $p$ of a poset $P$ the set
$$\dd p = \setof{ q \in P }{ q \leq p }$$ 
is  called the \dfn{principal ideal} generated by $p$.
We say that $F \subseteq P$ is an \dfn{initial subset} whenever for every $p \in P$: if $p \in F$ then $\dd p \subseteq F$.
For instance, if $A \subseteq P$
then $\dd A \eqdf \bigcup_{p \in A} \dd p$ is an initial subset of $P$. 

Also $\uu p \eqdf \setof{ q\in P }{ q \geq p }$ for $p \in P$ is called the \dfn{principal filter} generated by $p$,
and $\uu A \eqdf \bigcup_{p \in A} \uu p$ is called a \dfn{final subset} of $P$. 

A poset $P$ is called a \dfn{linear ordering} or a 
\dfn{chain} if the elements of $P$ are pairwise comparable.

A subset $A$ of $P$ is called an \dfn{antichain}
if $A$ consists of pairwise incomparable elements.

A poset $\pair{P}{\leq}$ is \dfn{well-founded}  if each nonempty subset of $P$ has a minimal element, that is $P$ has no strictly decreasing sequence. 
A \dfn{well-ordering} is a well-founded linear ordering.

Following \cite{DW}, we call a Hausdorff topological space $X$ a \dfn{Skula space} if the topology of $X$ is generated by the base $\setof{ U\setminus V }{ U,V\in\calTT }$ for some topology $\calTT$ on $X$. A typical example of a (non-scattered) Skula space is the Sorgenfrey line whose topology is generated by the base $\setof{ U\setminus V }{ U,V\in\calTT }$ for the topology $\calTT=\{\emptyset,\IR\}\cup\{(-\infty,a):a\in\IR\}$.

The following characterization of Skula spaces combines some known results of Bonnet and Rubin \cite{BR1} and of Dow and Watson \cite{DW} with some new results proved in this paper.
To state the theorem, following \cite[\S2.3]{BR1}, we introduce the following notion.
For a space $X$, we say that a family $\mU \eqdf \{U_x:x\in X\}$ is a \dfn{clopen selector} or more simple \dfn{selector} for $X$ if each $U_x$ is a clopen (i.e. closed and open) subset of $X$ and if $\mU$ satisfies:
\begin{itemize}
\item[(1)] $x\in U_x$ for every $x\in X$, 
\item[(2)] for any distinct $x,y\in X$ either $x\notin U_y$ or $y\notin U_x$, and 
\item[(3)] for any $x\in X$ and $y\in U_x$ we get $U_y\subseteq U_x$. 
\end{itemize}
Given that every $U_x$ contains $x$, conditions (2) and (3) are equivalent to: 
\begin{itemize}
\item[{\rm(4)}]
the relation
$$x <^{\mU} y \iff x\neq y \text{ and } x \in U_y$$ is irreflexive and transitive. 
\end{itemize}
Therefore a clopen selector $\mU \eqdf \setof{ U_x }{ x\in X }$ for $X$
induces a partial order relation $\leq^{\mU}$ on $X$,
defined by 
$$x \leq^{\mU} y \quad \text{if and only} \quad U_x \subseteq U_y \, .
$$
If $\mU$ is understood from the context, then $\leq^{\mU}$
is denoted by $\leq$ and thus
\smallskip
\newline
\smallskip
\centerline{
$U_x = \dd x$ 
}
is the clopen principal ideal for any $x \in X$:  
in particular every principal ideal is clopen. 

Skula spaces were introduced (independently of Skula) by Bonnet and Rubin \cite{BR1} in the algebraic way as ``well-generated Boolean algebras''.
For a compact space $X$, we denote by $\Clop(X)$ the set of clopen subsets of $X$. 

A \dfn{well-generated Boolean algebra} is a Boolean algebra  that has a sublattice  (that is, a subset closed under meet and join) that generates the algebra and that has no strictly decreasing sequence for $\leq$. 

It turns out that
$X$ is a compact Skula space if and only if $\Clop(X)$ is well-generated (see Theorem~\ref{thm-2.1}(i)$\Leftrightarrow$(iv) below).
Equivalently, $B$ is a well-generated Boolean algebra if and only if its space $\Ult(B)$
of ultrafilters is a Skula space. 
This equivalence was proved in a preprint of Bonnet and Rubin published in \cite[Proposition 2.15(b)]{BR1} and independently by Dow and Watson \cite[Theorem 1]{DW}. 
But Lemma~4 and Theorem~3 in \cite{DW} are based on a misquotation from \cite{BR1}. 

Clopen selectors (and stronger notions,
as ``canonical clopen selectors''),
and thus Skula spaces, were intensively studied
by U. Abraham, R. Bonnet, W. Kubi\'s and M. Rubin
in \cite{ABK, ABKR, BR1, BR2, BR8, BR3, BR4, BR5, BR6} in terms of Boolean algebras, but not so much as topological spaces. 

Relationships between ``being well-generated'' and ``being Skula'' can be found in \cite[\S2.3]{BR1} and in \cite[Theorem 3]{DW} and is stated by (i)$\Leftrightarrow$(ii) in the following result.

\begin{theorema}
\label{thm-1}
For a compact Hausdorff space $X$ the following conditions are equivalent. 
\begin{itemize} 
\item[{\rm(i)}]
$X$ is a Skula space. 
\item[{\rm(ii)}]
The Boolean algebra $\Clop(X)$ of clopen subsets of $X$ is well-generated. 
\item[{\rm(iii)}]
$X$ embeds into a compact join semilattice $H(X)$ such that the join operation is continuous and the set of principal ideals of $H(X)$ is a clopen selector for $H(X)$.
\end{itemize} 
\end{theorema}

This result will be proved in \S\ref{section-2}, and  
to show the part (i)$\Ra$(iii), we use Priestley spaces and their hyperspaces that we introduce now. 

\smallskip 
 
A \dfn{Priestley space}  
is a compact space having a partial order $\pair X \leq$ with the following separation property: 

\begin{itemize} 
\item[\BULLET]
For every $x, y \in X$: if $x \not\leq y$ there exists a clopen final subset $V$ of $X$ such that $x \in V$ and $y \notin V$. 
\end{itemize}
Therefore any Priestley space is $0$-dimensional. 

The notion of Priestley space was introduced by Hilary Ann  Priestley \cite{P} in 1970 to extend the duality in Boolean algebras to distributive lattices. 
For development of this notion, see the books of 
B. A. Davey and H. A. Priestley \cite[Ch. 11]{DP};
G.~Gierz, K.~Hofmann, K.~Keimel, J.~Lawson, M.~Mislove and D.~Scott \cite{GHKL1, GHKL2} and S. Roman \cite[Ch. 10]{R}.   
Trivially

\begin{itemize} 
\item[{}]
Any compact $0$-dimensional space can be
regarded as a Priestley space, where the partial ordering is the 
equality. 
\item[{}] 
Any Skula space is Priestley since any clopen selector  separates the points. 
 
\end{itemize}
In a Priestley space, if $\leq$ is understood from the context, then we omit it.
Now we shall discuss the notions of hyperspace over a Priestley space. 
The notion of hyperspace, over a compact space, was introduced by Leopold Vietoris \cite{V} in 1922.
For a detailed presentation of hyperspaces, we refer to 
E. Michael \cite{Michael}, the books of 
I. Illanes and S. Nadler~\cite{IS}, 
the section ``The exponential'' in the book of J. D. Monk~\cite[Ch. 1]{Monk}, 
 or in the handbook of K. P. Hart, J. Nagata and J.E. Vaughan~\cite[\S b-6 Hyperspaces]{HNV}. 
 We can see the hyperspace $H(X)$ over a Priestley space $X$ as a free join-semilattice over $X$. 
 For a general definition of free objects we refer to \cite{MacLane}.
 
\smallskip

Let $\pair X \leq$ be a Priestley space.
We define its  \dfn{(Vietoris) hyperspace} $H(X)$ 
as follows. 
\begin{itemize}
\item[\BULLET]
$H(X)$ is the set of all nonempty closed initial subsets of $\pair X \leq$. 

\item[\BULLET] 
For $F, G \in H(X)$, we set $F \leq G$ if and only if $F \subseteq G$. 

\item[\BULLET]
The topology $\calTT$ on $H(X)$ is the topology 
generated by the sets 
\\
$U^+ \eqdf \setof{ K\in H(X) }{ K\subseteq U }$ \quad and \quad 
$V^- \eqdf \setof{ K\in H(X) }{ K\cap V\neq \emptyset }$
where $U$ is any clopen initial subset and $V$ is any clopen final subset of $X$. 
Note that for any clopen initial subset $U$, setting $V = K \setminus U$, we have $H(X) \setminus U^+ = V^-$. 
\end{itemize} 
 
If $\leq$ is the equality on $X$, then $H(X)$, denoted by \dfn{$\exp(X)$},  is the set of all nonempty compact subsets of $X$, and is the well-known Vietoris hyperspace. 
Note that $H(X)$ is a subset of the compact space $\exp(X)$ and that $H(X)$ is compact (see Fact~\ref{fact-2.2} below) but not necessarily closed in $\exp(X)$, because the topology on $H(X)$ is not necessarily the induced topology of $\exp(X)$ (see Section~\ref{section-5.1}).

Section~\ref{section-2} is devoted to the proof of the following result that implies 
Part (iii)$\Ra$(vi) in Theorem~\ref{thm-2.1}. 

\begin{theorema}
\label{thm-2}
\begin{it} 
If $X$ be a Priestley space then  
\begin{itemize}
\item[{\rm(1)}]
$H(X)$ is a Priestley space (Theorem~{\rm\ref{thm-2.3}}).  
\item[{\rm(2)}]
$\pair{F}{G} \mapsto F \join G \eqdf F \cup G$ is a continuous semilattice  operation on $H(X)$. 
\item[{\rm(3)}]
$X$ is topologically embeddable in $H(X)$ by the increasing continuous map $\eta : x \mapsto \dd x$.
\end{itemize} 
\smallskip

\noindent
Moreover if $X$ is a Skula space then  $X$ is Priestley and
\begin{itemize}
\item[{\rm(4)}]
Any (nonempty) closed initial subset $K$ of $X$, i.e. a member of $H(X)$, is a finite union of clopen principal ideals of $X$. 
In particular $K$ is clopen
(Theorem~{\rm\ref{thm-2.5}}). 
\item[{\rm(5)}]
$H(X)$ is a Skula space. 
\item[{\rm(6)}]
The family $\setof{ K^+ }{ K \in H(X) }$ is a clopen selector for $H(X)$.
\end{itemize} 
\end{it}
\end{theorema}

\begin{remark*}
\label{KomentWithNoLabel}
Say that a distributive lattice $L$ is \emph{nice} whenever the the meet of two members of $L$ is the join of finitely many elements of $L$. 
For instance, by Theorem~\ref{thm-2}(4), if $X$ is Skula then $H(X) \cup \{\emptyset\}$ is a nice join-semillattice.
This phenomenon appears in \cite[Lemma~2.8(2)]{{BR1}} and 
\cite[Proposition 2.3(a1)]{ABK}.
\end{remark*}

Let us recall that a topological space $X$ is \dfn{scattered} 
if each nonempty subspace $A \subseteq X$ has an isolated point for the induced topology. 

For a given Skula space, to introduce its (Cantor-Bendixson) height $\height(X)$ and its (well-founded) rank $\rank(X)$, we need the following easy observation: for completeness we re-prove this fact. 

\begin{Fact}
\label{Fact-1}
\begin{it}
Let $\mU$ be a clopen selector of a Skula space for $X$. 
Then
\begin{itemize}
\item[{\rm(1)}]
$\pair{\mU }{\subseteq }$ is well-founded: see  \cite[Proposition 2.15(a)]{BR1}.
\item[{\rm(2)}]
Any minimal element $U_x$ of $\pair{\mU}{\subseteq}$ is of the form $\{x\}$ with $x \in X$: 
see the proof of Proposition {\rm2.7(b)} in \cite{BR1}. 

\item[{\rm(3)}]
$X$ is a scattered space: see \cite[Proposition 2.7(b)]{BR1}. 

\end{itemize}
\end{it}
\end{Fact}

\begin{proof} 
(1)
Suppose by contradiction that $U_{x_1} \supsetneqq U_{x_2} \supsetneqq U_{x_3} \supsetneqq \cdots$ is a strictly decreasing sequence of members of $\mU$. 
Let $F = \bigcap_n U_{x_n}$. 
For every $y \in F$, $U_y \subseteq F$, 
and thus $\bigcup_{y \in F} U_y = F$. 
Since each $U_y$ is open, $F$ is open. 
Hence each $U_{x_i} \setminus F$ is closed and nonempty, 
and $\bigcap_{i \in \omega} ( U_{x_i} \setminus F) = \emptyset$. 
A contradiction. 

(2) 
Let $U_x$ be a minimal element of $\mU$. 
We claim that $U_x = \{ x \}$. 
Otherwise there is $y \in U_x$ such that $y \neq x$.
By the definition, $x \not\in U_y$
and thus $U_y \subseteq U_x$. 
Therefore $U_x$ is not a minimal element of $\mU$. 
A contradiction. 

(3) 
By Part~(1), $X$ has isolated points. 
Next let $Y$ be a nonempty subspace of $X$. 
Then $\mV \eqdf \setof{ U_y \cap Y }{y \in Y}$ is a clopen selector for $Y$ and thus $Y$ has isolated points. 
\end{proof}

Now we introduce the ordinal invariants of Priestley and Skula spaces. 
For a subspace $A\subseteq X$ of a topological space $X$
denote by  $A^{[1]}$ the set of all non-isolated points in $X$.
Put $X^{[0]}=X$ and for every ordinal $\alpha>0$
define its
$\alpha$-th \dfn{(Cantor-Bendixson) derivative $X^{[\alpha]}$}
by the recursive formula: 
$$
X^{[\alpha]}=\mbox{$\bigcap$}_{\beta<\alpha}(X^{[\beta]})^{[1]}.
$$ 
It is easy to see that a topological space $X$ is scattered
if and only if $X^{[\rho]}=\emptyset$ for some ordinal $\rho$.
In this case
$X = \bigcup_{\alpha\leq\rho}\,  \bigl( X^{[\alpha]}\setminus X^{[\alpha+1]} \bigr) $. 
Therefore for each $x \in X$ we define the
\dfn{(Cantor-Bendixson) height}, denoted by $\height_X(x)$ or more simply by $\height(x)$, of $x$ in~$X$ by the formula:
$$
\height_X(x) = \alpha \quad \text{if and only if} \quad 
x \in X^{[\alpha]}\setminus X^{[\alpha+1]} \, .
$$ 
For instance $\height_X(x)=0$ if and only if $x$ is isolated in $X$.
The ordinal 
$$
\height(X)=\mbox{$\sup$}_{x\in X} \, \height_X(x) 
$$ 
will be called the
(Cantor-Bendixson) height of $X$. 
For example $\height(\omega^{\beta}+1)  = \beta$ for any ordinal $\beta \geq 1$. 
It follows that for a compact Hausdorff space $X$
the $\height(X)$-th derived set \dfn{$\Endpt(X) \eqdf X^{[\height(X)]}$} is finite and nonempty, and $\Endpt(X)$
is called the \dfn{set of end-points} of $X$. 
A scattered topological space $X$ is called
\dfn{unitary} if the set $X^{[\height(X)]}$
is a singleton, denoted by \dfn{$\{\lastpt(X) \}$} and
the element $\lastpt(X)$ is called the \dfn{end-point} of $X$. 

\smallskip 

Next let $\pair{W}{\leq}$ be a nonempty well founded poset. 
Therefore $\pair{W}{\leq}$ has a rank function
$\fnn{ \rank_W }{ W }{ \Ord }$ defined by  
$$
\rank_W(p) = \sup(\setof{ \rank_W(q)+1 }{ q < p }. 
$$
This ordinal, also denoted by $\rank(p)$, is called the \dfn{(well-founded) rank} of $p$ in $W$.  
For instance $\rank_W( p ) = 0$ if and only if $p \in \Min(W)$, i.e. $p$ is a minimal element of $W$.
The range of $\rank$, namely $\sup(\setof{ \rank_W(p)+1 }{ p \in W })$ is called the rank of $W$ and is denoted $\rank(W)$. 
Therefore $\rank_W(p) < \rank(W)$ for every $p \in W$.

Note that if $\fnn{ \rho }{ W }{ \Ord }$ is a strictly increasing function 
then $\rank_W(p) \leq \rho(p)$ for $p \in W$ and $\rank_W$ ``leaves no gap'': if $\gamma < \rank_W(p)$ then there is $q \in W$ such that $q < p$ and $\rank_W(q) = \gamma$.

Now assume that $X$ is a Skula space. 
By Theorem~\ref{thm-2}(5), $H(X)$ is also Skula and 
thus by Fact~\ref{Fact-1}(1), the posets $\pair{X) }{\leq}$ 
and $\pair{ H(X) }{ \subseteq }$ are well-founded.

In Section~\ref{section-3}, we show the following result.  

\begin{theorema}
\label{thm-3}
\begin{it}
Let $X$ be a Skula space and let $\mU$ be a clopen selector for $X$. 
Then
\begin{itemize}
\item[{\rm(1)}]
$\height(X) \leq \rank(X)< \w^{\height(X)+1}$ (Theorem~{\rm\ref{thm-3.1}}). 
\item[{\rm(2)}] 
$\rank(H(X)) \leq \w^{\rank(X)}$ 
(Theorem~{\rm\ref{thm-3.6}}). 
\end{itemize}
\end{it}
\end{theorema}

We introduce two stronger notions of Skula spaces developped in Section~\ref{section-4}.

\smallskip

We say that $X$ is a \dfn{canonical Skula space} whenever $X$ has a clopen selector $\mU \eqdf \setof{ U_x }{ x \in X }$ such that for every $x \in X$ the subspace $U_x$ is unitary and $\lastpt(U_x) = x$. 

Next we say that $X$ is a \dfn{tree-like canonical Skula space}, or more simply a \dfn{tree-like Skula space} whenever $X$ has a canonical clopen selector $\mU \eqdf \setof{ U_x }{ x \in X }$ satisfying: for $x, y \in X$,  $U_x$ and $U_y$ are either comparable or disjoint. 

By the definitions: 
$$
\text{tree-like Skula } \ \ \overset{\rm{(1)}}{\Longrightarrow}  \ \
\text{ canonical Skula }  \ \ \overset{\rm{(2)}}{\Longrightarrow} \ \
\text{ Skula }  \ \ \ \overset{\rm{(3)}}{\Longrightarrow}  \ \
\text{ Scattered}.
$$ 
None implication is reversible. 
For (1): take an uncountable almost disjoint $\mA$ on $\omega$ and consider its Mr\'owka space $\Psi(\mA)$ (see below for a precise definition).
For (2) see \cite[Theorem 3.25(a)]{BR1} and  
for (3), see \cite[Theorem 3.4]{BR1} or \cite[Example 1]{DW}. 
Note that any scattered and compact space of (Cantor-Bendixson) height $2$ is a canonical Skula space (Proposition~3.3(a) in~\cite{BR1}).

\medskip

First we develop some results on canonical Skula spaces. 

\begin{Fact} 
\label{Fact-2}
\begin{it}
Let $X$ be a Skula space and $\mU$ be a clopen selector for $X$. 
The following are equivalent. 
\begin{itemize}
\item[{\rm(i)}]
 $\mU$ is a canonical clopen selector for $X$. 
\item[{\rm(ii)}]
For any $x \in X$ the subspace $U_x$ is unitary and $\rank_{X}(x) = \height_{X}(x)$.
\end{itemize}
Therefore the well-founded rank $\rank_{X}(x)$ of $x \in X$ does not depend of the choice of canonical clopen selector for $X$. 
In particular $\rank(X) = \height(X)$. 
\end{it}
\end{Fact}

\begin{proof} 
Note that for every $z \in X$ we have 
$\rank_{X}(z) = \rank_{U_z}(z)$ 
and $\height_{X}(z) = \height_{U_z}(z)$.

We show (i) implies (ii) by induction on the (well-founded) rank.
If $\rank_{U_x}(x) = 0$  
then $x$ is minimal and $U_x = \{x\}$ and thus $\height_{U_x}(x) = 0$.  
Next assume that $\rank_{X}(y) = \height_{X}(y)$ for every $y < x$.  Since $U_x$ is unitary and $\lastpt(U_x) = x$, we have
\begin{eqnarray} 
\rank_{X}(x) 
&
= 
&
 \sup \setof{ \rank_{X}(y) {+}1 }{ y \in U_x {\setminus} \{x\} } 
\nonumber
\\
&
=
& 
\sup \setof{ \height_{X}(y) {+}1 }{ y \in U_x {\setminus} \{x\} } 
\, = \, \height_{X}(x).
\nonumber 
\end{eqnarray}
Therefore $\rank_{X}(x) = \height_{X}(x)$. 

We prove (ii) implies (i) by induction on $\rho(x) \eqdf \rank_{X}(x) = \rank(U_x)$.  
If $\rho(x) = 0$ then $U_x$ is finite and unitary, and thus $U_x = \{x\}$. 
Next fix $x \in X$. 
Suppose that $\lastpt(U_y) = y$ for every $y < x$. 
Since 
$\height(U_y) = \height_X(y) = \rank_{X}(x) = \rho(y) < \rho(x)$ 
and $U_y$ is unitary for every $y<x$, we have 
$\bigcup_{y<x} U_y = U_x \setminus \{x\}$. 
Now because $U_x$ is unitary, $x = \lastpt(U_x)$.
\end{proof}

We shall see in Section~\ref{section-4} (Theorem~\ref{thm-4.1}) the following main result. 

\begin{theorema}
\label{thm-4}
\begin{it}
If $X$ is a canonical Skula space then its hyperspace $H(X)$ is canonically Skula. 

Moreover, by Computation Rules~{\rm\ref{comment-4.5}}, we can calculate the Cantor-Bendixson height of any member of $H(X)$.
\end{it}
\end{theorema} 

\begin{Example}
\label{examplesa}
\begin{rm}
For instance let $U_\sigma \eqdf \bigcup_{x\in\sigma} U_x$ where  
$\sigma = \{ x_0, x_1, x_2, x_3, x_4, x_5, x_6, x_7 \}$ is an antichain of $X$ satisfying 
\smallskip
\newline
\begin{tabular}{ l l l}
$\height_X(x_0) = 0$
&
\quad$\height_X(x_1) = 1= \height_X(x_2)$ \nonumber
&
\quad$\height_X(x_3) = 2$ \nonumber
\smallskip
\\
$\height_X(x_4) = 10 =\height_X(x_5)$  
&
\quad$\height_X(x_6) = \omega+7$
&
\quad$\height_X(x_7) = 3$. \nonumber
\end{tabular}
\smallskip
\newline 
By Computation Rules~\ref{comment-4.5}, we have:
\smallskip
\newline
\phantom{\hfill$\blacksquare$}
\hfill
$\height_{H(X)}(U_\sigma) 
= \omega^{\omega+7} + \omega^9 \cdot 2 + \omega^2 +\omega + 2$. 
\hfill$\blacksquare$
\end{rm}
\end{Example}

\smallskip

The following result is proved in the same Section~\ref{section-4}:

\begin{theorema}
\label{thm-5}
\begin{it}
For any Skula space, the following are equivalent. 
\begin{itemize}
\item[{\rm(i)}]
$X$ is a tree-like canonical Skula space. 
\item[{\rm(ii)}]
$X$ is a continuous image of some successor ordinal space $\beta+1$.
\end{itemize} 
\end{it}
\end{theorema}

In \S\ref{section-4.1} we present some examples, 
namely spaces of initial subsets of a poset,  Mr\'owka spaces (defined by almost disjoint families), Lusin families, and ladder sequences.

The first examples of Skula spaces come from posets: 
let $P$ be a partial ordering. 
We denote by $\IS(P)$ the set of all initial subsets of $P$ (so $\emptyset, P \in \IS(P)$). 
Since $\IS(P) \subseteq \{0,1\}^P$, we endow $\IS(P)$ with the pointwise topology. 
We claim that 
\begin{itemize}
\item[]
$\IS(P)$ is a Priestley space in the pointwise topology $\calTT_p$\,.  
\end{itemize}
Indeed, obviously $\IS(P)$ compact. 
Next let $x, y \in \IS(P)$ be such that $x \not\subseteq y$. 
Pick $t \in x \setminus y \subseteq P$. 
Then $V_t \eqdf \setof{z \in \IS(P) }{ t \not\in z }$ is a clopen initial subset of $\IS(P)$ such that $y \in V_t$ and $x \not\in V_t$. 

In Section~{\rm\ref{section-4.1}} we characterize the posets $P$ such that $\IS(P)$ are Skula 
(such $P$'s are narrow and order-scattered {\rm\cite[Theorem 1.3]{ABKR}}\,). 
We give examples of posets $P$ such that $\IS(P)$ are Skula and are (or are not) tree-like canonically  Skula. 

Next recall that a family $\mA$ of infinite subsets of a set $S$ is called \dfn{almost disjoint} if $A\cap B$ is finite for any distinct sets $A,B\in\mA$. 
To eliminate trivialities, we assume that $\mA$ is infinite and that $\bigcup \mA = S$. 
A typical example of such a family $\mA$ is the uncountable set of branches of an ever-branching tree of height~$\omega$. A tree is \emph{ever-branching} if above every element one can find at least two incomparable elements.

A \dfn{Mr\'owka space}, also called a \dfn{$\Psi$-space}, is a unitary space of height 2. 
We describe a Mr\'owka space as follows.
Given an infinite almost disjoint family $\mA$ of subsets of a fixed set $S$, we define the compact space 
$K_\mA = S \cup \mA \cup \{ \infty \}$ as follows: 
all points of $S$ are isolated, a basic neighborhood of $A \in \mA$ is of the form $V_{A,F} := \{ A \} \cup (A {\setminus} F)$ with $F \subseteq A$ finite, and a basic neighborhood of $\infty$ is of the form $V_{\mS} \eqdf K_\mA \setminus \bigcup_{A \in \mS} V_{A,\emptyset}$ where $\mS$ is a finite subset of $\mA$.
In other words, $K_\mA$ is the space of ultrafilters of the Boolean subalgebra $B_{\mA}$ of $\powerset(S)$ generated by 
$\mA \cup \setof{ \{ x \} }{ x \in S} $ 
which is well-known under the name of \dfn{almost disjoint algebra}  over $\mA$. 

Spaces of the form $K_\mA$ appear often in the literature \cite{Mrowka},  although
actually they were introduced by Alexandrov and Urysohn, and they are well-known in Set-Theoretic Topology: see \cite[Ch.~3, \S11]{HBSTT}. 
Remark that $\SETOF{ \{ x \} }{x \in S } \cup \mA \cup \{ S \}$ is a canonical clopen selector for $K_{\mA}$. 
Note that
\begin{it} 
\begin{itemize}
\item[]
$K_{\mA}$ is a canonical Mr\'owka space 
and, by Theorem~{\rm\ref{thm-4.4}(3)}, $\height(H(K_{\mA})) = \omega$. 
\newline
Therefore $H(K_{\mA})$  is far from being a Mr\'owka space 
(because $\height(H(K_{\mA})) \geq 3$).
\\
However, by Theorem~{\rm\ref{thm-4}} and Fact~{\rm\ref{Fact-2}}, $H(K_{\mA})$ has a structure of canonically Skula space. 
\end{itemize} 
\end{it}
On the other hand, a Mr\'owka space $K_{\mA}$ is embeddable in a Mr\'owka space $G(K_\mA)$ with a continuous join operation (Theorem~\ref{thm-4.10}).
In particular $G(K_\mA)$ is unitary, of height~2, and thus of the form $K_{\mA^\star}$ (where $\mA^\star$ is an almost disjoint family). 
The statement above can be applied to ``Lusin families'' as constructed in \cite[Ch.~3, Theorem~4.1]{HBSTT}, and to ``ladder system'' as defined in Abraham and Shelah \cite{AS2}.  

\begin{it} 
\begin{itemize}
\item[]
There is a Lusin  family $\mL$ such that $K_{\mL}$ has a structure of continuous join operation.
\\
There is a ladder system $\mL$ such that $K_{\mL}$ has a structure of continuous join operation. 
\end{itemize}
\end{it}

\section{Priestley spaces and Vietoris hyperspaces}
\label{section-2} 

In this section we show Theorem~\ref{thm-1}, that we recall under the following form.

\begin{theorem}
\label{thm-2.1}
For a compact Hausdorff space $X$ the following conditions are equivalent. 
\begin{itemize} 
\item[{\rm(i)}]
$X$ is a Skula space. 
\item[{\rm(ii)}]
$X$ admits a partial order with clopen principal ideals. 
\item[{\rm(iii)}]
$X$ admits a closed partial order (as a subset of $X {\times} X$) with open principal ideals. 
\item[{\rm(iv)}]
The Boolean algebra $\Clop(X)$ of clopen subsets of $X$ is generated by a well-founded sublattice $\mW$. 
\item[{\rm(v)}]
The Boolean algebra $\Clop(X)$ of clopen subsets of $X$ is generated by a clopen selector. 
\item[{\rm(vi)}]
$X$ embeds into a compact semilattice $H(X)$ with clopen  principal ideals and the join operation on $H(X)$ is continuous. 
\end{itemize} 
\end{theorem}

\begin{proof}
The equivalence (i)$\Leftrightarrow$(v) was proved 
by Dow and Watson in \cite[Theorem 3]{DW},  
and (i)$\Leftrightarrow$(iv)$\Leftrightarrow$(v)$\Leftrightarrow$(i) by Bonnet and Rubin \cite[Lemma 2.8 and Proposition 2.15]{BR1}.

To see that (ii)$\Ra$(v), set 
$U_x = \dd x \eqdf \setof{ y \in X }{ y \leq x }$ for $x\in X$, 
and observe that 
$\mU \eqdf \setof{ U_x }{ x\in X }$ is a clopen selector.
The fact that the Boolean algebra $\Clop(X)$
is generated by the family $\mU$
is proved in \cite[Proposition 2.15(b)]{BR1}.

To see that (v)$\Ra$(ii), fix a clopen selector $\{U_x:x\in X\}$ and observe that it induces a partial order $\le$ with clopen principal ideals defined by $x \leq y$ iff $U_x \subseteq U_y$.

To see that (iii)$\Ra$(ii), observe that for $x \in X$, since $\leq$ is closed in $X {\times} X$, the set $\dd x \eqdf \setof{ y \in Y }{ y \leq x }$ is closed in $X$.

To see that (ii)$\Ra$(iii), it suffices to prove that each partial order $\le$ on $X$ with clopen principal ideals is closed in $X {\times} X$.
Let $\pair{x}{y}\in (X {\times} X) \, \setminus  \le$. 
Then $W \eqdf \dd y$ (that does not contains $x$) and $V \eqdf  \dd x \setminus \dd y$ are open in $X$, and thus $V \times W$ 
is a neighborhood of $\pair{x}{y}$, disjoint from the partial order relation $\leq$.

The implication (vi)$\Ra$(iii) is trivial, because the partial order $\leq$ on $H(X)$ is closed. 

The final implication (iii)$\Ra$(vi) follows from Theorem~\ref{thm-2.5}.
\end{proof}

To prove the part (iii)$\Ra$(vi), 
we use Priestley spaces and their hyperspaces. 

\smallskip 

A significant part of this section is devoted to the proof of Theorem~\ref{thm-2.1}(iii)$\Ra$(vi). 

By a \dfn{pospace (partially ordered space)}
we understand a topological space with a closed partial order $\leq$ (for the development of this concept, see for instance L. Nachbin \cite[Ch. 1]{Nachbin}). 

Note that any Priestley space, and thus any Skula space, is a partially ordered space.

A. Stralka \cite{Str} showed that there is a $0$-dimensional (non scattered) compact pospace which is not a Priestley space.
On the other hand, G. Bezhanishvili, R. Mines and P. Morandi \cite[Corollary 3.9]{BMM} proved that for any scattered and compact space:  the space is Priestley if and only if the partial order relation is closed.

In this work we assume that {\em every Priestley and thus every Skula space is compact and Hausdorff}.

\begin{fact}
\label{fact-2.2}
If $X$ is a Priestley space, then the space $H(X)$ is a Priestley space.
\end{fact}

\begin{proof} 
We denote by $\mV_X$ the set of clopen initial subsets of $X$. 
Since $X$ is Priestley, $\mS^0_X \eqdf \setof{U^+ }{U \in \mV_X }$ separates the points of $H(X)$ and thus $\pair{ H(X) }{ \calTT }$ is Hausdorff and $0$-dimensional.

We show that $H(X)$ is compact. 
Let $\mV_0 , \, \mV_1 \subseteq \mV_X$. 
We set $\mU_0 \eqdf \setof{U^+ }{U \in \mV_0 }$ and  
$\mU_1 \eqdf \setof{H(X) \setminus U^+ }{U \in \mV_1 }$. 
Assume that $\mU \eqdf \mU_0 \cup \mU_1$ is a centered family.

The set $A \eqdf \bigcap \mV_0$  is a closed initial subset in $X$.  
Moreover, for every finite subset $\mW$ of $\mV_0$ we have 
$\bigl( \bigcap \mW \bigr)^+ 
= \bigcap \setof{ V^+ }{V \in \mW} \neq \emptyset $. 
Since $V^+\neq\emptyset$ iff $V \neq\emptyset$, by the compactnes of $X$, $A \neq \emptyset$, and thus $A \in H(X)$.

We prove that $A \in \bigcap  \mU$.
By definition, $A \in \bigcap \setof{ V^+ }{V \in \mV_0} $. 
Fix $W \in \mV_1$ and suppose, by contradiction,
that $A \not\in H(X) \setminus W^+$.
The fact that $A \in W^+$ implies $ \bigcap \mV_0 \eqdf A \subseteq W$ and thus 
$\bigl( \bigcap \mV_0 \bigr) \setminus W = \emptyset$. 
Since $W$ is clopen and $X$ is compact,
there are $V_1,\dots,V_k \in \mV_0$ such that 
$V_1 \cap \cdots \cap V_k \subseteq W$.
This implies that
$V_1^+ \cap \cdots \cap V_k^+ \cap (H(X) \setminus W^+) = \emptyset$,
contradicting the fact that $\mU$ is centered.
It follows that $A \in H(X) \setminus W^+$ for every $W \in \mV_1$.
Finally, $A \in \bigcap \mU$. 
We have proved that $H(X)$ is compact. 

Next we check the Priestley separation property of $H(X)$.
Let $F, G \in H(X)$ be such that $G \not\subseteq F$. 
We fix $x \in G \setminus F$. 
For each $y \in F$, we have $y \not\geq x$, 
and thus there is a clopen initial subset $V_{y}$ of $X$ such that $y \in V_{y}$ and $x \not\in V_{y}$. 
By compactness, there is a finite subset $\sigma$ of $F$ such that $F \subseteq V \eqdf \bigcup_{y \in \sigma} V_{y} $ and $x \not\in V$, i.e. $F \not\subseteq V$.
Therefore $F \in U^+$ and $G \not\in U^+$. 

We have proved that $H(X)$ is Priestley.
\end{proof}

The reader might think that the compactness of $H(X)$ should follow from the compactness of the Vietoris hyperspace $\exp(X)$, which is a well known classical fact. Unfortunately, as we have already announced, $H(X)$ might not be a closed subspace of $\exp(X)$, see Example~\ref{prop-5.2} below.

Let $X$ be a Priestley space.
For $x \in X$ and  $A, B \in H(X)$ we set
$$\eta(x) \eqdf \dd x  
\quad \text{ and } \quad A \join B \eqdf A \cup B \, .
$$
So $\eta$ is a map from $X$ into $H(X)$ and $\join$ is a join semilattice operation on $H(X)$.  

\begin{theorem}
\label{thm-2.3}
Let $X$ be a Priestley space. 
Then

\begin{itemize} 
\item[{\rm(1)}]
$H(X)$ is a Priestley space. 

\item[{\rm(2)}]
The map $\,\fnn {\join} {H(X) \times H(X) }{ H(X) }$ is continuous.

\item[{\rm(3)}]
The map $\,\fnn { \eta } {X} { H(X) }$ is continuous and $x \leq y$ if and only if $\eta(x) \leq \eta(y)$.

\item[{\rm(4)}]
The join semilattice $\eta[X]^{\join}$ generated by
$\eta[X]$ in $H(X)$ is topologically dense in $H(X)$. 
\end{itemize}
\end{theorem} 

\begin{proof}
(1)
is proved in Fact~\ref{fact-2.2}. 

(2)
It is enough to notice that $F \cup G \in U^+$ if and only if $F \in U^+$ and $G \in U^+$. 

(3)
It is obvious that $\eta$ is an order-isomorphism.
To see that $\eta$ is continuous, it is enough to notice that $\inv \eta {U^+} = U$
for any clopen initial subset $U \subseteq X$.  

(4)
Recall that a nonempty basic open set in $H(X)$ is of the form
$V \eqdf U^+ \cap \bigcap_{i<n} ( H(X) \setminus W_i^+ )$
where $U, W_0, \ldots, W_{n-1}$ are clopen initial subsets in $X$.

For every $i<n$,  choose
$F_i \in U^+ \cap ( H(X) \setminus W_i^+)$. 
Since $F_i \cap (X \setminus W_i) \neq \emptyset$,  
pick $x_i \in F_i$ such that $x_i \in X \setminus W_i$. 
So $\eta(x_i) \eqdf \dd x_i \in U^+ \cap ( H(X) \setminus W_i^+ )$. 
Setting $\sigma = \setof{ x_i }{ i<n }$ 
we have $\dd \sigma \eqdf \bigcup_{i<n} \dd x_i \eqdf \Join_{i<n} \dd x_i \in L$ and $\dd \sigma \in V$.
We have proved that $\eta[X]^{\join}$ is topologically dense in $H(X)$.
\end{proof} 

We complete the above result on Priestley spaces by the following universal property. 
Recall that if $\pair Z \join$ is a compact $0$-dimensional join semilattice then $\Join B \eqdf \sup(B)$ exists for every nonempty subset $B$  of $Z$.

\begin{proposition}
\label{prop-2.4}
\begin{it}
Let $X$ and $\fnn{ \eta }{ X }{ H(X) }$ be as in 
Theorem~{\rm\ref{thm-2.3}}. 
Then for any $0$-dimensional compact join semilattice $Y$ and any increasing continuous map $\fnn{ f }{X }{ Y }$ there is a unique continuous join-semilattice homomorphism $\fnn{ \hat{f} }{H(X) }{ Y }$  such that $\hat{f} \circ \eta = f$. 
\end{it}
\end{proposition} 

\begin{proof} 
We show that the formula $\hat{f} (A) = \sup \img f A$ for $A \in H(X),
$
defines a continuous join-semilattice homomorphism $\hat{f}$
from $\pair{H(X)}{\join}$ into $\pair{Y}{\join}$. 

Fix $F, G \in H(X)$. Then 
$$
\hat{f} (F \join G) \eqdf \hat{f} (F \cup G)
= \sup \img f{F \cup G}
= (\sup \img f F) \join (\sup \img f G)
\eqdf \hat{f} (F) \join \hat{f} (G).
$$ 
Thus $\hat{f}$ is a join-homomorphism. 

To prove the continuity of $\hat{f}$ we shall use the following well-known fact. 

\smallskip

\noindent
\begin{claim*}%
Let $L$ be a compact $0$-dimensional space with a continuous semilattice operation $\join$, and let $W$ be a nonempty clopen initial subset of~$L$. 
Then $W$ is a finite union of clopen principal ideals of the form $\dd m$ where $m \in W$.
\end{claim*}

\begin{proof}
For every $x \in W$ let $C_x$ be a maximal chain in $L$ containing $x$. 
Since $W$ is clopen and $C_x$ is closed,  $C_x \cap W$ has a maximum $c_x$. 
Now fix any maximal element $c_x$ in $W$. 
Since the map  $\fnn{ f_{c_x} }{ L }{ L }$ defined by $f_{c_x} (t) = t \join c_x$ is continuous and since $c_x$ is maximal in $W$, 
\smallskip
\\
\smallskip
\centerline{
$\dd c_x 
\eqdf \setof{ t \in L }{ t \join  c_x = c_x} 
= \setof{ t \in L }{ t \join c_x \in W }
\eqdf f_{c_x}^{-1} [W]$
}
is clopen in $L$.
Finally by compactness, $W$ is a finite union of clopen principal ideals of the form $\dd c_x$ with $c_x \in W$. 
\end{proof}

For any clopen  principal ideal $V$ of $Y$ of the form $\dd p$ with $p \in Y$, $\inv f V$ is a clopen initial subset of $X$.
Let $K \in H(X)$. 
Then $\img f K \subseteq V$ iff $\sup \img f K \leq p$. 
In other words $K \in (\inv f V)^+$,  i.e.  $K \subseteq \inv f V$ 
iff $\hat{f}(K) \eqdf \sup \img f K \in  \dd p \eqdf V $. 
Therefore $\inv { \hat{f} } V = (\inv f V)^+$ and this set is clopen in $H(X)$. 

Now this fact and the above claim imply that $\hat{f}$ is continuous. This completes the proof of Proposition~\ref{prop-2.4}.
\end{proof}

The next result summarizes the properties of the hyperspace of a compact Skula space.
Given a subset $S$ of a poset, we denote by $\Max(S)$ the set of all maximal elements of $S$ and by $\Min(S)$ the set of all minimal elements of $S$.

\begin{theorem} 
\label{thm-2.5}
Let $X$ be a compact Skula space.
Then $X$ is a Priestley space and

\begin{itemize} 
\item[{\rm(1)}]
$H(X)$ is a Skula space.

\item[{\rm(2)}] 
Every closed initial subset $K$ of $X$ is clopen in $X$ and $K$ is finitely generated. 
\newline 
More precisely  $K = \dd \sigma_K$ where $\sigma_K \eqdf \Max(K)  \subseteq X$ is finite.

\item[{\rm(3)}]
$H(X)$ is the join semilattice generated by
$\setof{ \eta(x) }{ x \in X }$.

\item[{\rm(4)}]
$\setof{ K^+ }{ K \in H(X) }$ is a clopen selector for $H(X)$. 
\end{itemize}
\end{theorem}

\begin{proof} 
(2) 
Let $\mU \eqdf \setof{ U_x }{ x \in X }$ be a clopen selector for $X$.
Let $K$ be a nonempty compact initial subset of $X$. 
Since if $x \in K$ then $U_x \eqdf \dd x$ is a clopen subset of $X$ contained in $K$, 
by the compactness of $K$,
there is a finite subset $\sigma_K$ of $K$
such that $K = \bigcup_{z\in\sigma_K} U_z$.
Then $K = \dd K  = \bigcup_{z\in\Max(K)} U_z = \dd \Max(K)$.

(3) is a consequence of Theorem~\ref{thm-2.3} and of Proposition~\ref{prop-2.4}.

(4) 
By Fact~\ref{fact-2.2}, $H(X)$ is compact.
By Parts~(2) and~(3)
any $K  \in H(X)$ is a clopen initial subset $X$ and thus $K^+ \eqdf \setof{ L \in H(X) }{L \subseteq K }$
is clopen in $H(X)$.
Now it is easy to check that $\setof{ K^+ }{ K \in H(X) }$ is a clopen selector for $H(X)$. 

(1) follows from (4) and the fact that $H(X)$ is compact.
\end{proof} 

\section{Interplay between (Cantor-Bendixson) height and (well-founded) rank}
\label{section-3} 

The goal of this section is to show Theorem~\ref{thm-3}: 
(1):~if $X$ is a Skula space then $\height(X) \leq \rank(X)< \w^{\height(X)+1}$ and
(2):~$\rank(H(X)) \leq \w^{\rank(X)}$.

Concerning Part~(1) of Theorem~\ref{thm-3} we prove a little bit more:

\begin{theorem}
\label{thm-3.1}
For each Skula pospace $\pair{X}{\le}$ we have: 
$$
\height(X) \leq \rank(X) < \omega^{\height(X)} \cdot \big( | \Endpt(X)|+1 \big)
< \omega^{\height(X)+1}.
$$

\end{theorem}

In this theorem $\w^{\height(X)}$ is the $\height(X)$-th exponent of the ordinal $\w$.
The exponentiation of ordinals is defined by transfinite induction: 
$\alpha^0=1$ and $\alpha^{\gamma}=\sup_{\beta<\gamma}(\alpha^\beta\cdot \alpha)$ for $\gamma>0$. 
Recall that the multiplication of ordinals is also defined by 
transfinite induction: $\alpha\cdot 0=0$ and 
$\alpha\cdot\gamma = \sup_{\beta<\gamma}(\alpha\cdot\beta+\alpha)$ for $\gamma>0$.

\medskip

We cannot improve Theorem~\ref{thm-3.1} using ordinal invariants. 
Indeed 

\begin{itemize}

\item[{\rm(1)}] 
For the first inequality: if $X$ is canonically Skula then 
$\rank(X) = \height(X)$ 
(see Fact~\ref{Fact-2}). 
                    
\item[{\rm(2)}]
For the inequality $\rank(X) < \omega^{\height(X)+1}$, consider the ordinal space $X_n \eqdf \omega+n$ 
\hbox{($n \geq 1$).} 
Then $\height(X_n) = 1 < \rank(X_n) = \omega+n$ and $\sup_n \rank(X_n) = \omega^2 = \omega^{\height(X_n)+1}$. 
\end{itemize}

Theorem~\ref{thm-3.1} is a consequence of the following result.

\begin{lemma}
\label{lemma-3.2}
For any point $x\in X$ of a Skula pospace $X$ we have  
$$
\height_X(x) \leq \rank_X(x)
<\omega^{\height_X(x)}\cdot ( | \Endpt( \dd x) | + 1).
$$ 
\end{lemma}

\begin{proof}
The proof is divided in two parts. 

\smallskip

(1)
By induction on $\height(x)$,
we prove the inequality $\height(x) \le \rank(x)$.
This inequality is trivially true if $\height(x)=0$.
Assume that for some ordinal $\alpha>0$ the inequality
$\height(y) \le \rank(y)$ has been proved for all points $y\in X$
satisfying $\height(y)<\alpha$.
Choose any point $x\in X$ with $\height(x)=\alpha$.
Taking into account that $\dd x$ is clopen and thus open, observe that 
\begin{eqnarray} 
\height_X(x) 
&
=
& \min_{W}\sup_{y\in W\setminus\{x\}}(\height(y)+1)
\nonumber
\\
&
\le 
&\sup_{y\in \dd x\setminus\{x\}}(\height(y)+1)
 \ \le \sup_{y\in \dd x\setminus\{x\}}(\rank(y)+1) 
 \ \eqdf \rank_X(x)
\nonumber
\end{eqnarray} 
where the minimum, $\displaystyle \min_{W}$, is taken over all open neighborhoods $W$ of $x$ in $X$.

\smallskip 

(2) 
We shall prove 
$$\rank_X(x) \, < \, o_X(x) \eqdf \w^{\height(\dd x)} 
{\cdot} \bigl(  | \Endpt(\dd x) | {+} 1 \bigr)$$ 
by induction on the ordinal $o_X(x)$. 
By the definition of $o_X$, the map  
$\fnn{ o_X }{ \pair{X}{ \leq } }{ \pair{\Ord}{ \leq } }$ defined by $x \mapsto o_X(x)$ 
is increasing.

Assume first that  $o_X(x) \leq \omega$.
Since $| \Endpt(\dd x) | +1 \geq 2$ we have $\height(\dd x)=0$.
Therefore the set $\dd x$ is finite and hence
$\rank_X(x) \leq |\dd x| < |\dd x|+1 \eqdf o_X(x)$.

Next suppose that for some ordinal $\alpha$ \ --\,of the form $o_{.}(.)$\,--\  the inequality $\rank_T(t)<o_T(t)$ has been proved for all Skula
pospaces $T$ and all points $t \in T$ satisfying
$o_T(t) < \alpha$. 

Fix a point $x\in X$ such that $o_X(x) = \w^{\height(\dd x)} 
\cdot \bigl(  |  \Endpt(\dd x) | {+} 1 \bigr) \eqdf \alpha$. 
We set $E = \Endpt(\dd x)$ and, since $E$ is a finite poset, we fix $e \in \Min(E)$ i.e. $e$ is a minimal element of $E$. 
Let $Y = \dd e$ and $Z = (\dd x) \setminus (\dd e)$  ($Z$ can be empty).
So $Y$ and $Z$ are clopen subspaces of $X$, and $Y$ is unitary and $e = \lastpt(Y)$. 
Note that $Y = \dd e$ is an initial subset of $\dd x$ and that 
$\height(\dd e) = \height(Y) = \height(\dd x)$.

Fix $y \in Y$ with $y<e$. 
So $\height(\dd y) < \height(\dd e)$ and thus 
$\omega^{\height(\dd y)}  < \omega^{\height(\dd e)} 
= \omega^{\height(\dd x)}$.
Since $e \in \Min(E)$, by induction hypothesis, 
$\rank(y) < o_X(y) \eqdf \omega^{\height(\dd y)} (|E]+1) < \omega^{\height(\dd e)}$. 
This fact plus the fact that $Y$ is unitary imply:  

\begin{itemize}
\item[{\rm(i)}]
$\rank(Y) 
\eqdf \rank_Y(e) 
\leq \omega^{\height_X(\dd e)} = \omega^{\height_X(\dd x)}$.  
\end{itemize}
 
\noindent
{\it Case $1$. $|E| = 1$, i.e. $E = \{e\}$.}  
Since $e \in \Min(E)$ and $e \not\in Z \eqdf (\dd x) \setminus (\dd e)$ we have $\height(Z) < \height(Y)$, and thus  
by induction hypothesis, 

\begin{itemize}
\item[{\rm(ii)}] 
$\rank(Z) = \rank_Z(x) < o_Z(x) 
\eqdf \omega^{\height(Z)} \cdot ( | \Endpt(Z) | +1 )$
\newline
\phantom{$\rank(Z)$ }
$< \omega^{\height(Y)} = \omega^{\height(\dd x)}$. 
\end{itemize} 
Since $Y$ is an initial subset of $\dd x$ and thus $Z$ is a final  subset of $\dd x$, using (i) and (ii), we obtain: 
\begin{eqnarray} 
\rank(\dd x)  \leq \rank(Y + Z) 
&
\leq
&
\rank(Y)  +  \rank(Z) 
\ < \  
\omega^{\height(\dd x)} + \omega^{\height(\dd x)} 
\nonumber
\\
&
=
& 
\omega^{\height(\dd x)} \cdot 2 
\ \eqdf \ o_X(x) \ = \ \alpha
\nonumber 
\end{eqnarray} 
where $Y + Z$ denotes the lexicographic sum of $Y$ and $Z$, and so $y'<z'$ for every $y' \in Y$ and $z' \in Z$.

\medskip

\noindent
{\it Case $2$. $|E| \geq 2$.}
Recall that $Y \eqdf \dd e$ and 
$Z \eqdf (\dd x) \setminus (\dd e)$. 
Since $E \setminus \{e\}$ is nonempty, 
$\height(Y) \eqdf \height(\dd x) = \height(Z)$ and 
$\Endpt(Z) = E \setminus \{e\}$. 
So $ | \Endpt(Z) | = | E | -1 < | E |$. 
By the induction hypothesis, 

\begin{itemize}
\item[{\rm(iii)}]
$\rank(Z) < o_Z(x) \eqdf \omega^{\height(Z)} {\cdot} \bigl(  \, | \Endpt(Z) \,  | {+} 1 \bigr) 
= \omega^{\height(\dd x)} {\cdot} | E | $.   
\end{itemize} 
As in Case~1, by (i) and (iii), we obtain: 
\begin{eqnarray} 
\rank_X(x) \ = \ \rank(\dd x) 
\nonumber
&
\leq
&
\rank(Y)  +  \rank(Z) 
\nonumber
\\
 & <  & 
\omega^{\height(\dd  x))}  +  \omega^{\height(\dd x))} \cdot |E|
\smallskip
\nonumber
\\
&
\ < \ 
&
\omega^{\height(\dd  x)} {\cdot} \bigl(  | E | {+} 1 \bigr) \ \eqdf \  o_X(x) = \alpha \, , 
\nonumber 
\end{eqnarray} 
that ends the proof of the lemma. 
\end{proof}

\begin{proof}[Proof of Theorem~{\rm{\ref{thm-3.1}}}] 
By Lemma~\ref{lemma-3.2}, for $x \in X$ we have 
\begin{eqnarray} 
\height_X(x) \ \leq \ \rank_X(x)
\nonumber
&
<
&
\omega^{\height_X( \dd x ) } \cdot \bigl( \, \bigl| \, \Endpt(\dd x)  \, \bigr|+1 \bigr)
\ \eqdf \ \omega^{\height_X( \dd x )}\cdot m \  
\\
&
 <  
&
\omega^{\height(X)}\cdot \omega 
\ = \ \omega^{\height(X)+1}  
\nonumber 
\end{eqnarray} 
where $m \eqdf \bigl| \, \Endpt(\dd x)  \, \bigr|+1 < \omega$.
Hence $\height(X) \leq \rank(X) < \omega^{\height(X)+1}$\,:
the last inequality follows from the fact that there is an $x \in X$ such that
$\rank(X) = \rank(x)$.
\end{proof}

Our next step is to prove:
$$
\height(H(X)) \leq \rank(H(X)) \leq \w^{\rank(X)} < \w^{\w^{\height(X)+1}} \, .
$$
The first and the last inequality follows easily from Theorem~\ref{thm-3.1}. 
The difficult part is the inequality $\rank(H(X)) \leq \omega^{\rank(X)}$,  
for which we need some preparation.
Any (nonzero) ordinal $\alpha$ has a \dfn{Cantor decomposition}:
$\alpha \eqdf \omega^{\alpha_0} p_0 + \cdots + \omega^{\alpha_\ell} p_\ell$
where $\alpha_0 > \cdots > \alpha_\ell$ and $p_i > 0 $ for $i \leq \ell$.
We define the \dfn{Hessenberg's natural sum of ordinals}
(also called the \dfn{polynomial sum}) of the ordinals
$\alpha \eqdf \omega^{\gamma_0} p_0 + \cdots + \omega^{\gamma_m} p_m$ and
$\beta \eqdf \omega^{\gamma_0} q_0 + \cdots + \omega^{\gamma_m} q_m$
(where the $p_i$'s and $q_i$'s can be $0$), as the ordinal: 
$$
\alpha \oplus \beta
= \omega^{\gamma_0} (p_0+q_0) + \cdots + \omega^{\gamma_m} (p_m+q_m)\,.
$$ 
For example if $\alpha = \omega^{\omega + \omega} 8 + \omega^{7} 3$ and
$\beta = \omega^{\omega} + \omega^{7} + \omega^{2} + 5 $ then
$\alpha \oplus \beta = \omega^{\omega + \omega} 8 + \omega^{\omega}
+ \omega^{7} 4 + \omega^{2} + 5$.
Notice that $\oplus$ has the following properties:  
for every ordinals $\alpha$, $\beta$, $\gamma$ and $\delta$ we have 

\begin{itemize}
\item[{\rm(i)}]
$\alpha \oplus \beta = \beta \oplus \alpha$.

\item[{\rm(ii)}]
$(\alpha \oplus \beta) \oplus \gamma = \alpha \oplus (\beta \oplus \gamma)$.

\item[{\rm(iii)}]
$\alpha \oplus 0 = \alpha$.

\item[{\rm(iv)}]  
$\beta < \gamma$ if and only if $\alpha \oplus \beta < \alpha \oplus \gamma$.

\item[{\rm(v)}] 
$\alpha , \beta < \omega^{\delta}$ implies $\alpha \oplus \beta < \omega^{\delta}$.

\end{itemize} 
In \cite[Item (1), p. 55]{AB1} it is also shown that for every $\alpha$, $\beta$ and $\gamma$: 

\begin{itemize}
\item[{\rm(vi)}]
$\alpha \oplus \beta$ is strictly increasing in both arguments. 

\item[{\rm(vii)}]
if $\gamma < \alpha \oplus \beta$ then there are $\alpha' \leq \alpha$ and $\beta' \leq \beta$ such that $\gamma = \alpha' \oplus \beta'$ 
\\
(with  $\alpha' < \alpha$ or $\beta' < \beta$). 
(Do not be tempted to think that if 
$\alpha < \gamma < \alpha \oplus \beta$ then 
$\gamma = \alpha \oplus \beta'$ for some $\beta_0 < \beta$.)

\end{itemize} 
We give a useful application of the $\oplus$ operation due to 
Telg\`arsky \cite[Theorem 2]{T}. 
A proof can also be  found in Pierce
\cite[Ch. 21, Proposition 2.21.1]{Pi}. 

\begin{theorem}[Telg\`arsky]
\label{thm-3.3}
\begin{it}
Let $X$ and $Y$ be compact scattered spaces. 
For every $\pair{x}{y} \in X {\times} Y$ we have:
$\height_{X{\times}Y} (\pair{x}{y}) = \height_{X} (x) 
\oplus \height_{Y} (y) $, 
and thus $\height(X{\times}Y) = \height(X) \oplus \height(Y)$. 

Moreover if $X$ and $Y$ are unitary then
$X {\times} Y$ is unitary and
$$\lastpt(X{\times}Y) = \pair{ \! \lastpt(X) }{ \lastpt(Y)  }.$$ 
\end{it}
\end{theorem} 

Let $\pair{W}{\leq}$ be a nonempty well founded poset. 
We denote by 
$\Finnn{W}$ the set of nonempty finite subsets of $W$ 
and by $K(W)$ be the set of all initial subsets of $W$ generated by a nonempty finite subset $W$. 
That is, 
\begin{itemize}
\item[{\rm($*$)}]
$I \in K(W)$ if and only if there is 
$\sigma \in \Finnn{W}$ 
such that $I = \dd \sigma \eqdf \bigcup_{p \in \sigma} \dd p$. 
\\
Moreover we may assume that $\sigma$ is an antichain of $W$, and thus $\sigma = \Max(I)$.
\end{itemize}
By a result of Birkhoff \cite[Ch. VIII, \S2, Theorem 2]{Birkhoff}, $\pair{ K(W) }{ \subseteq }$ is well-founded. 
Therefore 
the rank function $\fnn{ \rank_{K(W)} }{ K(W) }{ \Ord }$ is well-defined.

The next result is due to 
N.~Zaguia~\cite[Ch. 1, Theorem~II-1.2]{Z} (Thesis, in French, 1983).
For completeness we recall his proof. 

\begin{theorem}[Zaguia]
\label{thm-3.4}
Let $\pair{ W }{ \leq }$ be a well founded poset.
Then $\rank(K(W)) \leq \omega^{\rank(W)}$.
\end{theorem}

\begin{proof} 
By the definition, 
$$\rank_{K(W)}(I) \eqdf \sup( \setof{ \rank_{K(W)}(J) + 1 }
{ J \in K(W) \text{ and } J \subsetneqq I })
$$ 
 for any $I \in K(W)$.  
For instance 
$\rank_{K(W)}(I) = 0$ if and only if there is a minimal element $p$ of $W$ such that $I = \{p\}$.
To prove the theorem, we need some preliminary  facts and the main key is the following result.
\medskip

\noindent 
\begin{it}%
Claim~$1$. 
Let $I \in K(W)$, and let $I' , I'' \in K(W)$ be such that $I = I' \cup I''$. 
Then 
\newline
\centerline{
$\rank_{K(W)}(I) \leq \rank_{K(W)}(I') \oplus \rank_{K(W)}(I'')$.
}%
\end{it}%

\proof 
The proof is done by induction on $\beta \eqdf \rank_{K(W)}(I)$. 
First if $\beta = 0$, i.e. $I = \{p\}$ with $p \in \Min(W)$, there is nothing to prove. 

Next let $I, I', I'' \in K(W)$ be such that $I = I' \cup I''$ with $\rank_{K(W)}(I) = \beta$. 
Let $J \in K(W)$ be such that $J \subsetneqq I$. 
Setting $J' \eqdf J \cap I'$ and $J'' \eqdf J \cap I''$ we have 
(1):~$J = J' \cup J''$, 
(2):~$J' \subseteq I'$ and $J'' \subseteq I''$ and  
(3):~$J' \neq I'$ or $J'' \neq I''$. 
Since $\rank_{K(W)}(J) < \beta \eqdf \rank_{K(W)}(I)$, by the induction hypothesis we have 
$$
\rank_{K(W)}(J) 
\leq \rank_{K(W)}(J') \oplus \rank_{K(W)}(J''). 
$$ 
Since $J' \neq I'$ or $J'' \neq I''$, 
\newline
\centerline{ 
$\rank_{K(W)}(J') \oplus \rank_{K(W)}(J'') 
< \rank_{K(W)}(I') \oplus \rank_{K(W)}(I'')$.
} 
Hence
\newline
\centerline{ 
$\rank_{K(W)}(J) 
< \rank_{K(W)}(I') \oplus \rank_{K(W)}(I'')
$} 
and thus $\rank_{K(W)}(I) \leq \rank_{K(W)}(I') \oplus \rank_{K(W)}(I'')$. 
\hfill$\blacksquare$
\medskip

\noindent
\begin{it}%
Claim~$2$.
For every $I \in K(W)$ 
$$
\rank_{K(W)} (I) \leq 
\rank_{K(W)} (\dd  p_{0})  \oplus \cdots \oplus  \rank_{K(W)} (\dd  p_{n-1})
$$ 
where $\{ p_{0},  \ldots, p_{n-1} \} \eqdf \Max(I)$. 
\end{it}

\proof 
Claim~2 follows from Claim~1 applied to 
$I = \bigcup_{i < n} \, ( \dd p_i )$
\hfill$\blacksquare$
\medskip

\noindent 
\begin{it}%
Claim~$3$.
For every $p \in W$ we have 
$\rank_{K(W)} (\dd p ) < \omega^{\rank_W(p)}$.
\end{it}

\proof 
By induction on $\beta \eqdf \rank_{W}(p)$.
First assume $\beta = 0$. 
So $p \in \Min(W)$ and thus $\{p\} = \dd  p \in \Min (K(W))$. 
Hence $\rank_{K(W)}(\dd p) = 0$, and consequently 
$\rank_{K(W)} (\dd p ) = 0 < 1 = \omega^0 = \omega^{\rank_W(p)}$. 

Now let $p \in W$ be such that $\beta = \rank_W(p)$. 
Let $I \in K(W)$ be such that $I \subsetneqq \dd x$. 
Let  
$\{ p_{0},  \ldots, p_{n-1} \} \eqdf \Max(I)$. 
For each $p_i < p$ we have $\rank_W(p_i) < \beta \eqdf \rank_W(p)$, and thus by the induction hypothesis 
$\rank_{K(W)} (\dd p_i ) < \omega^{\rank_{W}(p_i)} 
< \omega^{\rank_{W}(p)} \eqdf \omega^\beta$. 
Since $I = \bigcup_{i<n} \dd p_i$ and $\rank_{K(W)} (\dd p_i ) < \omega^{\rank_{W}(p)} \eqdf \omega^\beta$ for every $p_i$, by Claim~2, we have:
$$
\rank_{K(W)} (I) 
\leq \rank_{K(W)} (\dd p_0)  \oplus \cdots \oplus  \rank_{K(W)} (\dd p_{n-1}) 
< \omega^\beta \eqdf \omega^{\rank_{W}(p)} \, . 
$$  
Therefore 
$\rank_{K(W)} (\dd p) \leq \omega^{\rank_{W}(p)}$. 
\hfill$\blacksquare$

\medskip

Now we prove Zaguia's Theorem.
Let $I \in K(W)$. 
We set $\Max(I)  \eqdf \{ p_0,  \ldots, p_{n-1} \} $. 
By Claim~3, $\rank_{K(W)}(\dd p) < \omega^{\rank_W(p)}$ for any $p \in W$. 
So by Claim~2,  
$$
\rank_{K(W)} (I) \leq 
\rank_{K(W)} (\dd p_0)  \oplus \cdots \oplus  \rank_{K(W)} (\dd p_{n-1}) 
< \omega^{\rank(W)}\, . 
$$ 
Hence $\rank_{K(W)} (I) < \omega^{\rank(W)}$ for every $I \in K(W)$ and thus
$\rank(K(W)) \leq \omega^{\rank(W)}$. 
\end{proof}

\begin{remark}
\label{remark-3.5}
Let $X$ be a Skula space. 
Fix a clopen selector $\mU = \setof{ U_x }{ x \in X }$ for $X$. 
So $\mU$ defines a well-founded poset $\pair{ X }{ \leq }$ where we set $x \leq y$ if and only is $U_x \subseteq U_y$.
By Theorem~\ref{thm-2.5}(2), 
a member $G$ of $H(X)$ is an initial subset of $\pair{ X }{ \leq }$ generated by a nonempty finite subset of $X$. 
In other words $K(X) = H(X)$ is Skula and 
$\pair{H(X)}{\subseteq} \, = \, \pair{K(X)}{\leq}$.  
Therefore by Zaguia Theorem~\ref{thm-3.4} 
$\rank(H(X)) \leq \omega^{\rank(X)}$.
\hfill$\blacksquare$
\end{remark}

Now we can state the second result on the relationship between the rank and the height functions.

\begin{theorem}
\label{thm-3.6}
For any compact Skula pospace $X$,
its hyperspace $H(X)$ satisfies: 
$$
\height(H(X)) \leq \rank(H(X)) \leq \omega^{\rank(X)} < \omega^{\omega^{\height(X)+1}} \, .
$$ 
\end{theorem} 

\begin{proof}
By Theorem~\ref{thm-3.1},  
$\height(H(X)) \leq \rank(H(X))$, and   
$\rank(X) < \omega^{\height(X)+1}$ and thus  
$\omega^{\rank(X)} < \omega^{\omega^{\height(X)+1}}$. 
Finally the inequality $\rank(H(X)) \leq \omega^{\rank(X)}$ 
is proved in Remark~\ref{remark-3.5}. 
\end{proof}

Note that the above inequality $\rank(H(X)) \leq \w^{\rank(X)}$ cannot be improved:
consider the one-point compactification $X = \natN \cup \{\infty\}$ of the discrete space $\natN$ ordered by $n<\natN$ for $n \in \natN$. 
So $\natN$  is an antichain.
Then $X$ is a (tree-like) Skula space, $H(X) = [\Nat]^{<\omega} \cup \{X\}$ and thus $\rank(H(X)) = 1 = \omega^{\rank(X)}$.

\section{Canonical and tree-like Skula spaces and their applications}
\label{section-4}

Our first goal is to prove Theorem~\ref{thm-4} that we restate in Theorem~\ref{thm-4.1}. It is a consequence of Propositions~\ref{prop-4.3} and Theorem~\ref{thm-4.4}. 
We use only the following facts concerning a clopen selector $\calU$ for $X$:
\begin{itemize}
\item[$\BULLET$]
Any closed initial subset of $X$ is a finite union of members of $\calU$, and thus is clopen. 
In particular:
\item[$\BULLET$]
The intersection of any two members of $\calU$ is a finite union of members of $\calU$. 
\end{itemize}

We recall some notation. Given a set $I$, let 
\smallskip
\\
\smallskip
\centerline{
$\Fin{I}$  
be the set of finite subsets of $I$ \quad 
and \quad 
$\Finnn{I} = \Fin{I} \setminus \{ \emptyset \}$.
}

\begin{theorem}
\label{thm-4.1} 
\begin{it}
Let $X$ be a canonical Skula space.  
Then the space $H(X)$ is a canonical Skula space. 
More precisely $\setof{ U^+ }{ U \in H(X) }$ is a canonical clopen selector for $H(X)$.
\end{it} 
\end{theorem}

To prove the theorem, we need some preliminary facts. 
We fix a canonical selector $\mU \eqdf \setof{ U_x }{ x \in X }$ for  the Skula space $X$. 
So $U_x  \eqdf \dd x$. 
Also in what follows a ``finite and nonempty antichain'' of $X$ is abbreviated by ``antichain''. 

For a clopen selector $\mU$ of a Skula space $X$, $K \in H(X)$ and $x \in X$,

\begin{itemize}
\item[{\rm}]
$\height(K)$ denotes the height of the space $K$ (as compact subspace of $X$),  

\item[{\rm}]
$\height_{H(X)}(K)$ denotes the height of $K$ as element of $H(X)$, and 
\item[{\rm}]
$\rank_{X}(x)$ denotes the rank of $x$ as element of $X$, \dots 

\end{itemize}

\begin{lemma}
\label{lemma-4.2}
\begin{it}
Let $X$ be a canonical Skula space. 
Let $x \in X$ and $\sigma$ be an antichain of $X$ satisfying $x \not\in U_{\sigma}$. 
Then $\height(U_x \setminus U_{\sigma}) = \height(U_x)$.
\end{it}
\end{lemma} 

\begin{proof}
There is an antichain $\tau \subseteq U_x$ such  that 
$U_x \cap U_{\sigma} = U_{\tau}$. 
Hence $U_z \subsetneqq U_x$ for $z \in \tau$.
Since $U_x$ is unitary with end-point $x$ and $\tau$ is finite with $x \not\in U_\tau$, we have $\height(U_{\tau}) < \height(U_x)$. 
Hence $\height(U_x \setminus U_{\sigma}) = \height(U_x)$. 
\end{proof}

\begin{proposition}
\label{prop-4.3}
\begin{it} 
Let $X$ be a canonical Skula space,  
and let $\sigma$ be an antichain of $X$ such that $|\sigma| \geq 2$. 
Then 
\newline
\centerline{
$\height_{H(X)}(U_\sigma) = \bigoplus_{x \in \sigma} \height_{H(X)}(U_x)$.
}
\end{it}
\end{proposition}

\begin{proof} 
Let $\seqnn{ x_i }{ i<n }$ be an enumeration of $\sigma$ and for $i<n$ we set $U'_i = U_i \setminus \bigcup_{j<i} U_j$.
So for distinct $x, y \in \sigma$ we have $U'_x \cap U'_y = \emptyset$.   
By Lemma~\ref{lemma-4.2} $\height(U'_x) = \height(U_x)$. 
We define a map 
$$
\fnn{ f^{\sigma} \ }{ \ H(U_{\sigma} ) \ } \, \bigl(  
{ \  \mbox{$\prod$}_{x \in \sigma} \, H(U'_x) \cup \{ \emptyset \} \, \bigr)
\setminus \{ \vec \emptyset \} \  }
$$ 
where $ \vec \emptyset $ is the constant sequence  $\emptyset$, as follows. 
For $ K \in H(U_{\sigma}) $ we set  
$$ 
f^{\sigma}(K) = \seqnn{ K \cap U'_x }{ x \in \sigma }  \, . 
$$  
We allow
$K \cap U'_{x}  = \emptyset$. 
By convention, we set 
$\height(\emptyset) = 0 = \height_{H(X)}(\emptyset)$.

The map $f^{\sigma}$ is an order-isomorphism and onto, i.e. $K \subseteq L$ if and only if $K \cap U'_x \subseteq L \cap U'_x$ 
for each $x \in \sigma$ and a homeomorphism since $f^{\sigma}$ preserves infimum.    
By Telg\`arsky Theorem~\ref{thm-3.3}, 
$\height_{U_x'{\times}U_y'} (\pair{x'}{y'}) = \height_{U_x'} (x') \oplus \height_{U_y'} (y')$
for every $\pair{x'}{y'} \in U_x'{\times}U_y'$. 
Therefore 
 \begin{eqnarray}
\height_{H(X)}(U_\sigma) 
&
= 
&
\height_{\prod_{x \in \sigma} H(U'_x)}\seqnn{f^{\sigma}(U'_x)}{x\in\sigma}\nonumber
\\
&
=
&
\hbox{$\bigoplus$}_{x \in \sigma} \, \height_{H(X)}(U'_x) 
\ = \ \hbox{$\bigoplus$}_{x \in \sigma} \, \height_{H(X)}(U_x).
\nonumber
\end{eqnarray}
We remark that $\seqnn{K \cap U'_x}{x\in\sigma}$ depends on the enumeration $\seqnn{ x_i }{ i<n }$ of $\sigma$ but 
$\hbox{$\bigoplus$}_{x \in \sigma} \, \height_{H(X)}(U'_{x})$ does not depend of the enumeration of $\sigma$. 
\end{proof}

\begin{theorem}
\label{thm-4.4} 
\begin{it}
Let $\mU \eqdf \setof{ U_x }{ x \in X }$ be a canonical selector for $X$ and let $x \in X$. 
Then 
\smallskip
\newline
\centerline{
\smallskip
$\rank_X(x) = \height_X(x) = \height(U_x)$  and 
}
\begin{itemize}
\item[{\rm(1)}]
If $\rank_X(x) = 0$ then $\height_{H(X)}(U_x) = 0$. 
\item[{\rm(2)}]
If $\rank_X(x) = 1$ then $\height_{H(X)}(U_x) = 1$. 
\item[{\rm(3)}] 
If $\rank_X(x) = 1{+}\alpha \geq 2$ then 
$\height_{H(X)}(U_x) = \omega^{\alpha}$. 
\end{itemize}
Moreover for any $x \in X$ the subspace $H(U_x) = U_x^+$ of $H(X)$ is unitary and $\lastpt( H(U_x) ) = U_x$.  
\end{it}
\end{theorem}

\begin{proof}
Since each $U_x$ is unitary with end-point $x$, 
by Fact~\ref{Fact-2}, 
$\rank_X(x) = \height_X(x) = \height(U_x)$.
So we compute only $\height_{H(X)}(U_x)$ by induction on $\rank_X(x) = \height_X(x)$. 

We need to recall that the Hessenberg product $\odot$ of ordinals is defined as follows (see \cite{AB2}) . 
Let $\alpha$ and $\beta$ be ordinals. 
We set  
\begin{itemize}
\item[]
\qquad\qquad
$\alpha \odot 0 = \alpha$,
\item[]
\qquad\qquad
$\alpha \odot (\beta+1) = (\alpha \odot \beta) \oplus \beta$ and 
\item[]
\qquad\qquad
$\alpha \odot \beta = \sup_{\gamma<\beta} \alpha \odot \gamma$ for a limit~$\beta$. 
\end{itemize}
Operation $\odot$ is not the same as the Hessenberg multiplication $\alpha \otimes \beta$ which is obtained from the normal forms of $\alpha$ and $\beta$ viewed as polynomials and multiplied
accordingly. In particular, $\alpha \otimes \beta$ is commutative, but $\odot$ is not.  
For instance 
\begin{itemize}
\item[{}]
\qquad\qquad
$\omega \otimes 2 = 2 \otimes \omega = \omega+\omega$
\item[{\rm}]
\qquad\qquad
$\omega \odot 2 = \omega \oplus \omega = \omega + \omega$ and 
$2 \odot \omega = \sup_{n<\omega} 2 \odot n = \omega$. 
\end{itemize}
The function $\alpha \odot \beta$ is strictly increasing in the right variable, continuous in the right variable, and non-decreasing in the left variable. 
Obviously 
\begin{itemize}
\item[{\rm}]
\qquad\qquad
If \ $n \in \omega$ \  then \ 
$\omega^\alpha \cdot n \ = \ \omega^\alpha \odot n \ = \ \omega^\alpha \otimes n$. 
\end{itemize}
It is easy to check that 
$\alpha \cdot \beta \leq \alpha \odot \beta \leq \alpha \otimes \beta$ where $\alpha \cdot \beta$ is the usual operation.

\medskip

\noindent
{\it Case $1$.}
\begin{it}%
$\rank_X(x) = 0$. 
\end{it}%
So $ U_x = \{ x \} $ and thus $\height_{H(X)}(U_x) =  0$. 
Obviously $H(U_x) = \{ x \}$ and thus $H(U_x)$ is unitary with $\lastpt(H(U_x)) = \{ x \} = U_x$. 
\medskip

\noindent
{\it Case $2$.}
\begin{it}%
$\rank_X(x) = 1$. 
\end{it}%
The set $U_x \setminus \{x\}$ is infinite and discrete. 
Recall that $H(U_x)$ is the set of all nonempty initial subsets of $U_x$, considered as subspace of $X$.
It is easy to see that the set 
of all nonempty antichains contained in $U_x \cap \Min(X)$ is the set of isolated points of $H(U_x)$. 
Hence $H(U_x)$ is homeomorphic to the one-point compactification of a discrete space and $\lastpt(H(U_x)) = U_x$.  Therefore $\height_{H(X)}(U_x) = 1$. 

\medskip 

\noindent
{\it Case $3$.}
\begin{it}%
$\rank_X(x) = 1 + \alpha$ with $\alpha = 1$.
\end{it}%
The set 
\smallskip
\\
\smallskip
\centerline{ 
$S_1 \eqdf \setof{ y \in X }{ y<x \text{ and } \rank_X(y) = 1 }
\ = \ \setof{ y \in U_x \setminus \{x\} }{ \rank_X(y) = 1 \, }$ 
}
is infinite. 
In particular if $y < x$ and  $y \not\in S_1$ then $\rank_X(y) = 0$.
Let $\sigma$ be an antichain of $U_x$ such that $x \not\in \sigma$. 
Since $\bigcup_{y\in\sigma} U_y \eqdf U_\sigma \subseteq U_x$, by Proposition~\ref{prop-4.3},  
 \begin{eqnarray}
\height_{H(X)}(U_\sigma) 
&
= 
&
\mbox{$\bigoplus$}_{y\in \sigma \cap S_1} \, \height_{H(X)}(U_y) 
\oplus 
\mbox{$\bigoplus$}_{y\in \sigma \setminus S_1}  \, \height_{H(X)}(U_y) 
\nonumber
\\
&
=
&
1 \odot | \sigma \cap S_1 | \oplus 0 \ = \ | \sigma \cap S_1 |.
\nonumber
\end{eqnarray}
In particular if $\sigma$ is an antichain contained in $S_1$ 
then  
\smallskip
\\
\smallskip
\centerline{
$\height_{H(X)}(U_\sigma)  \ =  \ \bigoplus_{y\in\sigma} \height_{H(X)}(U_y) = 1 \odot |\sigma|  \ = \  |\sigma|.$
}
Now, since $S_1$ is infinite, 
it is easy to see that $\height_{H(X)}(U_x) = \omega \eqdf \omega^1$. 
Obviously $H(U_x)$  is unitary and $\lastpt(H(U_x)) = U_x$.

\medskip

\noindent
{\it Case $4$.}
\begin{it}%
$\rank_X(x) = 1+\alpha$ with $\alpha \geq 2$. 
\end{it}%
Fix $\beta < \alpha$ with $\beta \geq 1$. 
The set 
\smallskip
\\
\smallskip
\centerline{
$S_\beta = \setof{ y \in X }{ y<x \text{ and } \rank_X(y) = \beta } 
\ = \ \setof{ y \in U_x \setminus \{x\} }{ \rank_X(y) = \beta \, }$ 
} 
is infinite. 
Now let $\sigma$ be an antichain contained in $S_\beta$ and 
$U_\sigma \eqdf \bigcup_{y\in\sigma} U_y$\,. 
So $U_\sigma \subseteq U_x \setminus \{x\}$. 
Remark that for every $y \in \sigma$
\smallskip
\\
\smallskip
\centerline{
$\beta = \rank_X(y) \leq \gamma 
\eqdf \max \setof{ \rank_X(z) }{ z \in \sigma } < \alpha$.
} 
For every $y \in S_\beta$,
we have   
$\beta = \rank_X(y) \leq \gamma < \alpha$. 
Moreover by Fact~\ref{Fact-2}, 
$\rank_X(y) = \height_X(y) = \height(U_y)$.  
In other words, $\beta = \height(U_y) \leq \gamma$. 

Fix $y \in S_\beta$. 
Since $1+\beta \geq 2$, by the induction hypotheses, we have 
$\height_{H(X)}(U_y) = \omega^\beta$. 
(Note that if $ 1 + \beta =\alpha = 2$ then by Case~3, 
$\height_{H(X)}(U_y) = \omega \eqdf \omega^1 = \omega^\beta$.) 
As in Case~3, by Proposition~\ref{prop-4.3}, 
\smallskip
\\
\smallskip
\centerline{
$\height_{H(X)}(U_\sigma)  \ 
=  \ \bigoplus_{y\in\sigma} \height_{H(X)}(U_y) 
=  \ \omega^{\beta} \odot |\sigma|  \ 
=  \ \omega^{\beta} \cdot |\sigma|  \ 
<  \ \omega^{\alpha}$. 
} 
Now since $S_\beta$ is infinite, $\beta$ is any ordinal (strictly) less than $\alpha$ and $|\sigma|$ is arbitrary, it is easy to check that 
\smallskip
\\
\smallskip
\centerline{
$\height_{H(X)}(U_x) 
= \sup \setof{ \, \omega^{\beta} \cdot m \, }{ \, \beta < \alpha \text{ and } m \in \omega \, } = \omega^{\alpha}$.
} 
So $H(U_x)$ is unitary and 
$\lastpt(H(U_x)) = U_x$ because $\height_{H(X)}(U_\sigma) < \omega^{\alpha} = \height_{H(X)}(U_x)$ 
for any finite antichain $\sigma$ of $U_x$ satisfying $U_\sigma \subsetneqq U_x$.

\smallskip

We have proved Theorem~\ref{thm-4.4}.
\end{proof} 

\begin{proof}[Proof of Theorem~{\rm\ref{thm-4.1}}]
By Theorem~\ref{thm-2.5}(4), the set of 
$U^+ \eqdf \setof{ K\in H(X) }{ K\subseteq U }$, where $U$ is any clopen initial subset of $X$, defines a clopen selector for $H(X)$. 
Therefore it suffices to prove that each $U^+$ is unitary and clopen subset of $H(X)$ with $\lastpt(U^+) = U$.    

We fix a closed (equivalently clopen) initial subsets  $U$ and $V$ of $X$. 
So $U = U_{\sigma}$ and $V = U_\rho$ where $\sigma$ and $\rho$ are antichains of $X$.  
Assume that $U_\rho \subsetneqq U_\sigma$. 
It suffices to show that  
\smallskip
\newline
($\star$) \hfill  $\height(U_\rho) < \height(U_\sigma)$.
\hfill \phantom{($\star$)}
\smallskip
\newline
We show ($\star$) by induction on $|\sigma|$.
If $|\sigma| = 1$, then we are done by Theorem~\ref{thm-4.4}.  
Next assume that $ |\sigma| \geq 2$. 
Notice that 
$U_\rho = U_\rho \cap U_\sigma = \bigcup_{s\in\sigma} (U_\rho \cap U_s)$ and that, by compactness, each $U_\rho \cap U_s$ is a finite union of $U_z$'s. 
For any $s \in \sigma$ we have $U_\rho \cap U_s \subseteq U_s$. 
By Theorem~\ref{thm-4.4},
$\height(U_\rho \cap U_s) \leq \height(U_s)$ for $s \in \sigma$. 

Since $U_\rho \subsetneqq U_\sigma$, there is $s' \in \sigma$ such that $U_{\rho} \cap U_{s'} \subsetneqq U_{s'}$. 
Again, by Theorem~\ref{thm-4.4}, $\height(U_{\rho} \cap U_{s'}) < \height(U_{s'})$. 
But the function $\pair{\alpha}{\beta} \mapsto \alpha \oplus \beta$ is strictly increasing in both arguments. 
Therefore by Proposition~\ref{prop-4.3},
\smallskip
\\
\smallskip
\centerline{
$\height(U_\rho) 
= \bigoplus_{s\in\sigma} \height(U_\rho \cap U_s) 
< \bigoplus_{s\in\sigma} \height(U_s) 
= \height(U_\sigma)$.
}
This ends the proof of Theorem~\ref{thm-4.1}.
\end{proof}

\begin{algorithm}
\label{comment-4.5}
\begin{rm}
For a nonempty antichain $\sigma$ of $X$, by Theorem~\ref{thm-4.1}, the subspace $H(U_\sigma) = U_\sigma{}^+$ of $H(X)$ is unitary with 
$\lastpt( H(U_{\sigma}) ) = U_{\sigma}$, and    
we can calculate $\height_{H(X)}(U_\sigma)$: 
\begin{itemize}
\item[-]
by Proposition~\ref{prop-4.3} we have 
$\height_{H(X)}(U_\sigma) = \bigoplus_{x \in \sigma} \height_{H(X)}(U_x)$, and
\item[-]
by Theorem~\ref{thm-4.4} we know 
$\height_{H(X)}(U_x)$ in function of $\height_{X}(x)$ for any $x \in X$. 
\end{itemize}
Such a calculation appears in Example~\ref{examplesa}.
\hfill$\blacksquare$  
\end{rm}
\end{algorithm} 

Next we develop properties of tree-like canonical Skula spaces: 
Proposition~\ref{prop-4.6},
 (i)$\Leftrightarrow$(ii) was proved U. Abraham and R. Bonnet~\cite{AB1} and 
 (ii)$\Leftrightarrow$(iii) is due to R. Bonnet and H. SiKaddour \cite[\S 2.4 and \S2.6]{bonnet}. 
 
The ``Moreover'' is due to M. Rubin: 
we recall that a topological space $Z$ is \dfn{retractable} whenever every nonempty closed set $F$ of $Z$ is a retract, i.e. there is a continuous map $\fnn{ f }{ Z }{ F }$ such that $f {\restriction} F$ is the identity on $F$. 

For completeness we give a sketch of the proof of the next result, simplifying the techonology of the original proof.

\begin{proposition}
\label{prop-4.6}
Let $X$ be a compact space.
The following are equivalent.

\begin{itemize}
\item[{\rm(i)}]
$X$ is a scattered space and
$X$ is a continuous image of a $0$-dimensional complete linear ordered space $\pair{X}{\leq}$ endowed with the order  topology.
\item[{\rm(ii)}]
$X$ is a continuous image of a successor ordinal $\alpha+1$
(endowed with the order topology).
\item[{\rm(iii)}]
$X$ has a tree-like canonical clopen selector.
\end{itemize} 
Moreover if $X$ is a continuous image of a successor ordinal then $X$ is retractable. 
\end{proposition} 

\begin{proof}[Sketch of Proof]
Let us begin by an observation.
Let $\alpha$ be an infinite ordinal. 
Then $\setof{ [0, \beta] }{ \beta \leq \alpha }$ is a clopen selector for $[0, \alpha]$. 

On the other hand, as in the proof of Theorem~\ref{thm-3.1},  any non-zero ordinal $\beta$ has a Cantor decomposition:
$\beta \eqdf \omega^{\beta_0} p_0 + \cdots + \omega^{\beta_\ell} p_\ell$ 
where $\beta_0 > \cdots > \beta_\ell$ and $p_i \geq 1 $ for $i \leq \ell$. 
Denote by  $\tip(\beta)$ the last block.
For instance, if $\beta = \omega^{\omega} \cdot 2 + \omega ^ 7 + \omega^{3} \cdot 5$ then $\tip(\beta) = \omega^{3}$ and if $\beta = \omega+5$ then $\tip(\beta) = 1$.
So $[\tip(\beta), \beta)$ is order-isomorphic to $\tip(\beta) \eqdf \omega^{\beta_\ell}$. 
Therefore $\tip(\beta)+1$ and thus 
$U_\beta \eqdf (\tip(\beta), \beta]$ are unitary and of Cantor-Bendixson height $\beta_\ell$. 
Now, obviously the set of 
$\SETOF{  \! U_\beta }{\beta \leq \alpha }$ 
is a tree-like canonical clopen selector for $[0,\alpha]$.

\smallskip

(ii)$\Rightarrow$(iii) \cite[\S2.6]{bonnet}.
The space $X$ is a continuous image of $\alpha+1$, which means that 
the Boolean algebra $\Clop(X)$ is a (superatomic) subalgebra of $\Clop([0, \alpha])$. 
We recall this construction of the tree-like canonical selector $\mU$ for $X$.
For any $U \in \Clop(X)$ there is a unique finite strictly
increasing sequence $\vec{s}\,{}^U \eqdf \seqn{ s_i^U }{ i < 2\ell(U) }$
of members of $\alpha+1$ such that:
$U = \bigcup_{i<\ell} (s_{2i}^U , s_{2i+1}^U]$ with
$(s_{2i+1}^U, s_{2i+2}^U] \neq \emptyset$ for $i < \ell(U)-1$.
Fix $x \in X$. 
Let 
\begin{eqnarray}
m[x]
&
=
&
\min  \, \SETOF{  \, \ell(U) \in \omega \, }
{ \, U \in \Clop(X)  \text{ and }  \Endpt(U) = \{ x \} \, } \, ,
\mbox{\rm\ \ \ \  and}
\nonumber
\\
\mU[x]
&
=
&
\SETOF{ U \in \Clop(X) \, }{  \, \Endpt(U) = \{ x \} \text{ and }   \ell(U) = m[x] \, } \, .
\nonumber
\end{eqnarray}
We recall that for any integer $n$ (in particular if $n \eqdf m[x]$), 
the set $[\alpha+1]^{n}$ of finite strictly increasing sequences of $\alpha+1$ of length $n$ is well-ordered by the lexicographic order relation denoted by $\preceq$. 
Hence 
$$
S[x] \eqdf \setof{ \vec{s}\,{}^U }{ U \in \mU[x] }
$$
is well-ordered by $\preceq$, and thus 
$\vec{t}\,^U \eqdf \min^{\preceq}( S[x] )$ exists. 
We set $$U_x = \bigcup_{i<m[x]} (t_{2i}^U , t_{2i+1}^U].$$

Then $ \mU \eqdf \setof{ U_x }{ x \in X } $ 
is the required tree-like canonical clopen selector for $X$:  
see~\cite[\S2.6: Part C]{bonnet}. 

\smallskip

(iii)$\Rightarrow$(ii) \cite[\S2.4]{bonnet}. 
We use the following fact, whose proof can be obtained by induction on the (well-founded) rank $\rank(X)$ of $X$: 
for analogous results see \,S. Todor\v cevi\'c in \cite[Ch 6 \S 2]{HBSTT} and S. Koppelberg \cite[Ch. 6  \S 16]{kopp}. 
Let $\mU \eqdf \setof{ U_x }{ x \in X }$ be a tree-like canonical selector for $X$, considered as a well-founded set of subsets of $X$. 
So $x \leq y$ if $U_x \subseteq U_y$.  
There are a well-ordering $\preceq$ on $X$ and a one-to-one map 
$\fnn{ \varphi }{ \mU }{ \powerset(X) }$, satisfying that for every $x, y \in X$: 

\begin{itemize}
\item[{\rm(1)}]
If $x \leq y$ then $x \preceq y$,

\item[{\rm(2)}]
$\varphi(U_x)$ is an half-open interval in $\pair{X}{\preceq}$ of the form $(a_x, b_x]$ with $a_x, b_x \in X$, 

\item[{\rm(3)}]
If $U_x = \{x\}$ then $\varphi(U_x)$ is a singleton, and 

\item[{\rm(4)}]
$U_x \subseteq U_y$ iff $\varphi(U_x) \subseteq \varphi(U_y)$, 
and 
$U_x \cap U_y = \emptyset$ iff $\varphi(U_x) \cap \varphi(U_y) = \emptyset$. 

\end{itemize} 
By Sikorski's extension theorem \cite[Theorem 5.5]{kopp}, 
we extend $\varphi$ in a one-to-one Boolean homomorphism from $\Clop(X)$ into the interval algebra $B(X)$ over $\pair{X}{\preceq}$.

\smallskip

(i)$\Leftrightarrow$(ii). See Abraham and Bonnet~\cite[Theorem 1]{AB1}.

\smallskip

The ``Moreover'' part is a re-statement of a result of M. Rubin \cite[Theorem 5.1]{Ru} proved in terms of Boolean algebras.
\end{proof}

\begin{comment*}
\label{comment-no-label-1}
\begin{rm}%
M. Pouzet \cite{pouzet} gave the following proof of Proposition~\ref{prop-4.6}(iii)$\Rightarrow$(i). 
Let $X$ be a set and let $\mF$ be a family of nonempty subsets of $X$ such that two members of $\mF$ are either comparable or disjoint.
Then $\mF$ is order-isomorphic to a set of intervals of a linear ordering. 
He proved first the case where $\mF$ is finite and then made the final conclusion using the ``Compactness Theorem''.
\end{rm}
\end{comment*}

\subsection{The space of initial subsets of a partial ordering}
\label{section-4.1}

Let $P$ be a poset.
Recall that $\IS(P)$ denotes the set of all initial subsets of $P$ 
(so $\emptyset, P \in \IS(P)$). 
Let $\FS(P)$ be the set of all final subsets of $P$. 
Then $\fnn{ \varphi }{ \IS(P) }{ \FS(P) }$ defined by $\varphi(I) = P \setminus I$ is an isomorphism between the complete distributive lattices $\pair{ \IS(P) }{ \subseteq }$ and $\pair{ \FS(P) }{ \supseteq }$.

Since $\IS(P), \FS(P) \subseteq \{ 0,1 \}^P$, we endow $\IS(P)$ and $\FS(P)$ with the pointwise topology. 
Hence the spaces $\IS(P)$ and $\FS(P)$ are compact and $\varphi$ is a homeomorphism. 
So: 
\begin{itemize}
\item[{\rm($\star$)}]
{\em We identify the Priestley spaces $\pair{ \IS(P) }{ \subseteq }$ and $\pair{ \FS(P) }{ \supseteq }$ endowed with the pointwise topology $\calTT_p$.}
\end{itemize}
In \cite[Theorem 2.3]{ABKR}, it is shown:
\begin{itemize}
\item[{($\star\star$)}]
The Boolean algebra $\Clop(\FS(P))$ of clopen subsets of $\FS(P)$ is the poset algebra $F(P)$. 
\end{itemize} 
For the definition and properties of (free) poset algebras $F(P)$ see \cite{ABKR} and \cite[\S3]{ABK}. 

\begin{remark}
	Let $X \eqdf \alpha+1 = [0,\alpha]$ be an infinite successor ordinal, considered as a Skula space. 
	Then $H(X) = \IS(X) \setminus \{ \emptyset\}$ and it is easy to verify that
	the spaces $\IS(X)$ and $H(X)$ are homeomorphic. 
\end{remark}

We say that a poset $P$ is \dfn{narrow} 
if every antichain (set of pairwise incomparable elements) is finite. 
A poset $P$ is \dfn{order-scattered} if $P$ does not contain a copy of the rational chain $\natQ$.
The next result can be found in \cite[Theorem 1.3]{ABKR}. 

\begin{proposition}
\label{prop-4.8}
\begin{it}
Let $P$ be a poset. 
The following are equivalent. 
\begin{itemize}
\item[{\rm(i)}]
$P$ is a narrow and order-scattered poset. 
\item[{\rm(ii)}]
$\FS(P)$ is a scattered space, i.e. the poset algebra $F(P)$ is superatomic, i.e. every quotient algebra of $F(P)$ has an atom.  
\item[{\rm(iii)}]
$\FS(P)$ is a Skula space, i.e. the poset algebra $F(P)$ is well-generated. 
\qed
\end{itemize}
\end{it}
\end{proposition}

A poset $P$ is a \dfn{well-quasi ordering (w.q.o.)} whenever $P$ is  narrow and well-founded. 
The notion of w.q.o. was introduced by G. Kurepa in 1937, cited in \cite{Kurepa}, and is a frequently discovered concept: see for instance Kruskal \cite{Kruskal}. 
We recall two facts for which the proof is obvious.

\begin{proposition}
\label{prop-6.2}
\begin{it}
Let $P$ be a partial ordering. 
The following are equivalent.
\begin{itemize}
\item[{\rm(i)}]
$P$ is a well-quasi ordering. 
\item[{\rm(ii)}] 
$\pair{ \IS(P) }{ \subseteq }$ (i.e. $\pair{ \FS(P) }{ \supseteq }$) has no strictly decresing sequence. 
\item[{\rm(iii)}]
Any nonempty final subset $K$ of $P$ is finitely generated, i.e. $K$ contains a nonempty finite subset $\sigma$ such that $K = \uu \sigma$. 
\qed 
\end{itemize} 
\end{it} 
\end{proposition}

At the opposite of Proposition~\ref{prop-4.8}(i)$\Rightarrow$(iii), for which the proof is far from being obvious, the proof in some special case is quite trivial. 

\begin{proposition}[Special case of Proposition~\ref{prop-4.8}]
\label{prop-6.3}
\begin{it}
Let $P$ be a well-quasi ordering. 
Then $\FS(P)$ is a Skula space, and thus $\FS(P)$ is a scattered space. 
\end{it}
\end{proposition}

\begin{proof}
Obviously $\FS(P)$ is compact. 
By Proposition~\ref{prop-6.2}, 
for any $K \in \FS(P)$ 
there is a nonempty finite antichain $\sigma_K$ in $P$ such that $K = \uu \sigma_K$. 
Therefore  
\smallskip
\newline
\smallskip
\centerline{
$U_K^+ 
\eqdf \setof{ F \in \FS(P) }{ \sigma_F \subseteq G } 
= \setof{ F \in \FS(P) }{ F \supseteq K } $ 
}
is a clopen subset of $\FS(P)$. 
It is obvious to see that  
$\mU \eqdf \setof{ U_K^+ }{ K \in \IS(P) }$ 
is a clopen selector for $\FS(P)$.  
\end{proof} 

Note that we do not know if $\FS(P)$ is canonically Skula whenever $P$ is a well-quasi ordering: see  
Questions~\ref{question-6.4}--\ref{question-6.8}. 

Also let us remark that in special cases we can say more than in Proposition~\ref{prop-4.8}:  
Part~(1) of the next result seems to be well-known, but we could not find it in the literature.

\begin{proposition}
\label{prop-4.9}
\begin{it}
\begin{itemize}
\item[{\rm(1)}]
If $P$ is an order-scattered linear ordering, 
then the space $\FS(P)$ is a quotient of a successor ordinal, and thus $\FS(P)$ has a tree-like canonical clopen selector. 
\item[{\rm(2)}] 
If $P$ is the disjoint union of two copies of $\omega_1$ then 
$\FS(P) \cong (\omega_1+1)^2$ is canonically Skula but $\FS(P)$ has no tree-like canonical clopen selector. 
\end{itemize}
\end{it}
\end{proposition}

\begin{proof}
(1)
We set $C = \FS(P)$. 
So $C$ is a complete chain, i.e. every subset of $C$ has a supremum and an infimum, and $C$ is a topological scattered space. 
Remark that the pointwise topology on $\FS(P)$ is the order topology on $C$. 

Since $C$ is order-scattered, between any two elements of $C$ there are two consecutive elements and thus $C$ is $0$-dimensional. 
We prove the claim by induction on $\height(C)$. 
If $\height(C) = 0$, $C$ is finite and there is noting to prove. 
Next suppose that $\height(C) = \alpha$. 
We assume that for every complete and scattered chain $D$: 
if $\height(D) < \alpha$ then $D$ is a continuous image of some ordinal $\delta+1$.
Since $C$ is $0$-dimensional, it suffices to prove the result whenever $C$ is unitary. 

We set $c^0 = \min(C)$ and $c^1 = \max(C)$.
The point $e = \lastpt(C)$ is called two-sided if $[c^0,e)$ has no maximum and $(e,c^1]$ has no minimum.
We claim that we may assume that $e$ is not two sided. 
Indeed, otherwise we split $e$, that is we replace $e$ by two consecutive elements $e^- < e^+$. 
So we obtain a chain $\wh{C} = [c^0, e^-] + [e^+, c^1]$ satisfying $\height(\wh{C}) = \height(C)$ and $\emptyset \neq \Endpt(\wh{C}) \subseteq \{ e^- , e^+ \}$. 
The identification of $e^-$ with $e^+$ defines an increasing and continuous map from $\wh{C}$ onto $C$. 
Hence it suffices to prove the result whenever 
$C \eqdf [c^0, e^-]$ and $\lastpt(C) = e^-$. 
So $\max(C) = e^- \eqdf \lastpt(C)$.   
The case $C \eqdf [e^+, c^1]$ is similar.

Let $\seqnn{ c_\alpha }{ \alpha<\lambda}$ be a strictly increasing and unbounded sequence in $[c^0, e^-)$. 
Since $C$ is complete, we may assume that $c_0 = c^0$ and that 
$\sup_{\beta<\alpha} c_\beta =  c_\alpha$ for every limit $\alpha<\lambda$. 
For each limit $\alpha <\lambda$ we add an immediate successor $d_\alpha$ to $c_\alpha$ that is: $d_\alpha \not\in C$  and for every $x \in C$:
if $x > c_\alpha$ then $x > d_\alpha$. 
Hence we obtain a chain 
$\overline{C} = C \cup \setof{ d_\alpha }{ \alpha < \lambda }$. 
The identification of $d_\alpha$ with $c_\alpha$ for all $\alpha$, defines an increasing and continuous map from $\overline{C}$ onto $C$. 
So it suffices to prove the result for $\overline{C}$ and thus, 
we may assume that $C = \overline{C}$.

We set $C_0 = [c_0, c_1]$ ($c_0 \eqdf \min(C)$), and for each successor $\alpha \geq 1$ let $C_\alpha = [d_\alpha , c_{\alpha+1}]$. 
Since $C_0 = [c_0, d_1)$ and $C_\alpha = (c_\alpha , d_{\alpha+1})$, each $C_\alpha$ is a clopen subset of $C$.
Also for every limit $\alpha$ we set $C_\alpha = \{ c_\alpha \}$ 
(recall that $c_\alpha$ has a successor $d_\alpha$ in $C$).  
So $C$ is the lexicographic sum 
$\bigl(  \sum_{\alpha<\lambda} C_\alpha \bigr) + \{ \max(C) \}$. 
Now since $\height(C_\alpha) < \height(C)$, by the induction hypothesis, 
there is a successor ordinal $\delta_\alpha$ and a continuous function $f_\alpha$ from $\delta_\alpha$ onto $C_\alpha$. 
If $\alpha$ is limit, and thus $C_\alpha = \{ c_\alpha \}$, we may assume that $\delta_{\alpha} = 1$. 
Moreover we set $f_\lambda(\max(C)) = \max(C)$. 
Hence $f \eqdf \bigcup_{\alpha \leq \lambda} f_\alpha$ is a continuous map from $\bigl(  \sum_{\alpha<\lambda} \delta_\alpha \bigr) + \{ \max(C) \}$ onto the well-ordering 
$C \eqdf  \bigl(  \sum_{\alpha<\lambda} C_\alpha \bigr) + \{ \max(C) \}$.

(2) 
Let $P = \omega_1 \, \sqcup \, \omega_1$ be the disjoint union of two copies of $\omega_1$, that is 
$P \eqdf \omega_1 {\times} \{0\} \cup \omega_1 {\times} \{1\} $ and $x$ and $y$ are incomparable for any $x \in \omega_1 {\times} \{0\}$ and $y \in \omega_1 {\times} \{1\}$. 
So $\IS(P) \cong \IS(\omega_1) \times \IS(\omega_1) \cong (\omega_1+1)^2$. 
By Telg\`arsky Theorem~\ref{thm-3.3}, the product of two unitary canonical Skula spaces is canonically Skula.
So $(\omega_1+1)^2$ is a unitary canonically Skula space. 

Now, by contradiction, assume that $(\omega_1+1)^2$ has a tree-like canonical clopen selector. 
By Proposition~\ref{prop-4.6}, $(\omega_1+1)^2$ is a quotient of $\alpha+1$ for some ordinal $\alpha$ and $(\omega_1+1)^2$ is retractable. 
But it is obvious that $(\omega_1+1)^2$ is not retractable: consider the closed subset 
$(\omega_1{+}1) {\times} \{\omega_1\} \cup \{\omega_1\}  {\times} (\omega_1{+}1)$ of $ (\omega_1 {+} 1)^2$.
A contradiction. 
\end{proof}

\subsection{Mr\'owka spaces}
\label{section-4.2} 

Recall that a Mr\'owka space $K_{\mA}$ is a unitary canonical Skula space of height~2. 
The space $K_{\mA}$ is defined by an infinite almost disjoint family $\mA$ on an infinite set $S$. 
We may assume that $\lastpt(K_{\mA}) = \max(K_{\mA})$
\begin{itemize}
\item[{\rm(1)}]
Let $\mA$ be an infinite almost disjoint family on $S$.
Then the space $H(K_\mA)$ is a unitary canonical Skula space of height~$\omega$
(reformulation of Theorems~\ref{thm-4.1} and~\ref{thm-4.4}(3) with $1+\alpha=2$.
Therefore $H(K_\mA)$ is far from being  a Mr\'owka space. 
\item[{\rm(2)}]
Let $\mA$ be a maximal almost disjont family on $\omega$.
Then the Mr\'owka space $K_{\mA}$ is not homeomorphic to a topological semilattice (Proposition~\ref{prop-5.3}). 
\end{itemize} 

We describe, in two ways, a general procedure of modifying an almost disjoint family $\mA$ on a set $S$ leading to a Mr\'owka space $K_{\mA^\star}$ with a continuous join operation.
Recall that $\Finnn{I}$ denotes the nonempty and finite subsets of $I$.

On one hand, given $A, B \in \mA$ with $A \neq B$, notice that
$\Fin A \cap \Fin B = \Fin{A\cap B}$ is finite.
Setting $A^\star \eqdf \Finnn{A}$, it follows that the family 
$$
\mA^\star = \setof{ \, A^\star }{ A \in \mA  \, } 
$$
is almost disjoint on $S^\star \eqdf \Fin S$. 
Therefore 
$$K_{\mA^\star} \eqdf S^\star \cup \mA^\star \cup \{\infty^\star\}
$$ 
is a Mr\'owka space, where $\infty^{\star} = \Max(K_{\mA^\star})$. 

On the other hand, we can describe $K_{\mA^\star}$ is a more formal way as follows.
Since $K_\mA$ is canonically Skula, by Theorem~\ref{thm-4}, 
$H(K_{\mA})$ is a unitary canonical Skula space.
We may assume that $\infty^{\star} \eqdf \lastpt(H(K_{\mA})) = \max(H(K_{\mA}))$.
Any member $L \eqdf \dd L$ of $H(K_{\mA})$ with  $L \neq K_\mA$ is of the form
$L = \bigcup \SETOF{ A \cup \{ x_A \} }{ A \in \mA_L } \cup \rho_L$ where $\mA_L \varsubsetneqq \mA$ is finite, $\rho_L$ is a finite subset of $S$ and $ | \mA_L | + | \rho_L | \geq 1$.
We set 
$$
E = \SETOF{ L \in H(K_\mA) }{ | \mA_L | \geq 2  \text{ or} , 
| \mA_L | = 1 | \text{ and } | \rho_L | \geq 1 }
\, .
$$ 
Obviously $E$ is a closed final subset of $\pair{ H(K_\mA) }{ \subseteq }$: 
indeed $H(K_\mA) \setminus E = \bigcup \setof{ K^+ }{ | \mA_K | \leq 1 }$ is an open initial subset of $H(K_\mA)$.
Hence the set $E$ induces the closed equivalence relation 
$$
\E = \setof{ (x,y) \in H(K_\mA) {\times} H(K_\mA) }{ x=y }
\cup ( E \times E)
$$ 
on $H(K_\mA)$. 
Let   
$$
G(K_\mA) = H(K_\mA) / \E  . 
$$
Since  we collapse only  all elements of $\E$ in a point, denoted by $\infty$, we have:
$$
G(K_\mA) = \SETOF{ L \in H(K_\mA) }
{  L \in \Finnn S \text{ or } L \in \mA } \cup \{\infty\}. 
$$ 
We denote by $\fnn{ q }{ H(K_\mA) }{ G(K_\mA) }$ the quotient map. 
Obviously $G(K_\mA)$ is compact. 
For each $L \in G(K_\mA) \setminus \{\infty\}$, $\height(L) \leq 1$, and thus $G(K_\mA)$ is of height~2 and unitary. 
So $G(K_\mA)$ is a Mr\'owka space. 
Moreover $q(L) = L$ for any $L \in K_{\mA^{\star}}) \setminus \{\infty^{\star}\}$, and $q(\infty^{\star}) = \infty \eqdf E$.
So, identifying $\infty^{\star}$ with $\infty$, 
\begin{itemize}
\item[{\rm($\star$)}]
The identity map $\fnn{ {\rm Id} }{ K_{\mA^{\star}} }{ G(K_\mA) }$ (with $\infty^{\star} \mapsto \infty$) is a homeomorphism. 
\end{itemize}

To show that $K_{\mA^{\star}} \eqdf G(K_\mA)$ has a structure of a continuous join operation $\join$, we need the following fact that can be found in \cite[Theorem 1.54]{CHK1}. 

\medskip

\noindent
\begin{it}%
Claim.
Let $\pair{Y}{m_Y}$ be a compact topological join  semilattice and  
let $E$ be a closed nonempty final subset of $Y$. 
Then the quotient space $X \eqdf Y/E$ obtained by identification of all points of $E$ has a continuous join operation $m_X$. 
\qed
\end{it} 
\medskip

Since $H(K_\mA)$ is a compact $0$-dimensional join semilattice, by the claim, $G(K_\mA)$ has a continuous join semilattice operation. 
We have proved the following result.

\begin{theorem}
\label{thm-4.10}
Let $K_\mA$ be a Mr\'owka space. 

\begin{itemize}
\item[{\rm(1)}]
$K_{\mA^{\star}} = G(K_\mA)$ and 
$G(K_\mA)$ is a Mr\'owka space with a continuous join operation and $G(K_\mA)$ has a canonical selector. 

\item[{\rm(2)}]
$\fnn{ \eta }{K_\mA }{ K_{\mA^{\star}} }$ defined by 
$\eta(x) = \dd x$ for $x \in K_\mA$ 
is a one-to-one, increasing and continuous function. 
\qed
\end{itemize} 
\end{theorem} 

Now we will apply the above results to some examples of Mr\'owka spaces.

\subsection{Lusin families and ladder systems}
\label{section-4.3} 
An uncountable almost disjoint family $\mA$ of infinite subsets of  $\Nat$
is called a \dfn{Lusin family} 
(called inseparable family by Abraham and Shelah in~\cite{AS1})
if $\bigcup \mA = \Nat$ and
for any subset $H \subseteq \Nat$ one of the families
$\setof{ A \in \mA }{A \subseteq^* H }$
or $\setof{ A \in \mA }{ A \subseteq^* X {\setminus} H }$ is countable.
Here we denote by $\subseteq^*$ the almost inclusion: for two sets $A, B$ we write \dfn{$A\subseteq^* B$} 
if $A \setminus B$ is finite.
The first example of a Lusin family was constructed by Lusin \cite{Lusin}
who actually constructed a ``special Lusin family''. 
For completeness we give the proof of~Proposition~\ref{prop-4.11} (cf.~\cite[Ch.~3, Theorem~4.1]{HBSTT}).

\begin{proposition}[Lusin]
\label{prop-4.11}
There exists a Lusin family $\mL$ on $\Nat$ of cardinality $\aleph_1$.
\end{proposition}

\begin{proof}
We construct by transfinite induction pairwise almost disjoint sets
$A_\alpha \in \dpower \Nat \omega$ so that for each
$\alpha < \omega_1$ the following condition is satisfied: 
\begin{equation}
\label{Eqjeden}
\bigl( \, \forall k\in\omega\,\bigr)\;\; \bigl( \, |  \,
\{\xi<\alpha: A_\xi \cap A_\alpha \subseteq k\} \, | < \aleph_0 \, \bigr) \, .
\tag{L1}
\end{equation} 
We start by choosing arbitrary disjoint infinite sets
$A_0, A_1, \ldots \subseteq \Nat$.
Fix $\omega \leq \beta < \omega_1$ and suppose $A_\xi$ have been constructed for $\xi < \beta$.
Enumerate $\setof{A_\xi}{\xi<\beta}$ as $\seqof{B_n}{\ntr}$.
Construct $A_\beta$ in such a way that 
\begin{equation}
\label{EqDwaa}
\bigl( \,\forall \ntr \, \bigr)\;\; \bigl( \, |A_\beta \cap B_n \setminus (B_0 \cup \cdots \cup B_{n-1})| = n \, \bigr) \, .
\tag{L2}
\end{equation}
It is clear that (\ref{Eqjeden}) holds.
Thus, the construction can be carried out.

We claim that $\El = \seqnn{A_\alpha}{\alpha < \omega_1}$ is a Lusin family. 
By contradiction, suppose that $H \subseteq \Nat$ is such that both sets 
$$
L = \setof{\alpha < \omega_1}{A_\alpha \subseteq^* H}
\qquad\text{and}\qquad
R = \setof{\beta < \omega_1}{A_\beta \subseteq^* \Nat \setminus H}
$$ 
are uncountable.
So $H$ is infinite.
Refining $L$ and $R$, we may assume that for some $k \in \omega$ the inclusion
\begin{equation}
\label{EqTrzi}
A_\alpha \setminus H \subseteq k \qquad\text{and}\qquad A_\beta \cap H \subseteq k
\tag{L3}
\end{equation}
holds for every $\alpha \in L$ and $\beta \in R$.

Choose $\beta \in R$ so that the set $L \cap \beta$ is infinite.
Then, by (\ref{Eqjeden}), for each $k$ we can find
$\xi(k) \in L \cap \beta$ such that $A_{\xi(k)} \cap A_\beta\not\subseteq k$.
Choose $x_k \in A_{\xi(k)} \cap A_\beta \setminus k$.
Then, by (\ref{EqTrzi}), we conclude that $x_k \in H$
and therefore, since $x_k \in A_\beta$, we have $x_k \in A_\beta \cap H$.
Hence $A_\beta \cap H$ is infinite, contradicting (\ref{EqTrzi}).
\end{proof}

\begin{comment*}
\label{comment-no-label-2}
\begin{rm}%
The crucial property of the almost disjoint family invented by Lusin is Condition (\ref{Eqjeden}).
A family $\mA$ of infinite subsets of a countable set $N$ is called a
\dfn{special Lusin family} 
if it satisfies condition~(\ref{Eqjeden}) with the quantifier ``\,$(\forall\; n\in \Nat)$\," replaced by ``\,$(\forall\; s \in \Fin N)$\,".
That is: for each $\alpha < \omega_1$\,: 
\begin{equation}
\label{Egwida}
\forall\; s\in\Fin N)\;\; | \setof{\xi < \alpha}{ A_\xi \cap A_\alpha \subseteq s} | < \aleph_0.
\tag{{\small $\star$}}
\end{equation} 
Obviously, this property depends on the enumeration of the family.

Therefore the above proof shows the existence of a special Lusin family. 
Also Lusin's Theorem says that a special Lusin family is a Lusin family.
\end{rm}
\end{comment*} 

The next results follows from Theorem~\ref{thm-4.10}.

\begin{proposition}
\label{prop-4.12}
\begin{it} 
Let $\mL$ be a special Lusin family on $\natN$.
\begin{itemize}
\item[{\rm(1)}]
$\mL^\star$ is a special Lusin family.

\item[{\rm(2)}]
$G(K_\mL) = K_{\mL^\star}$ and $K_{\mL^\star}$ admits a continuous join semilattice structure. 
\qed
\end{itemize}
\end{it}
\end{proposition}

If $L \subseteq \omega_1$  is a stationary set, then a 
\dfn{ladder system} over $L$ is a sequence $\El = \seqnn{ c_\alpha }{ \alpha< \omega_1 }$ ($\alpha \in L$ and $\alpha$ is a limit ordinal) such that each 
$ c_\alpha \eqdf \seqnn{ c_{\alpha, n} }{ n< \omega }$ is a strictly increasing $\omega$-sequence cofinal in $\alpha$. 
So $\El$ is an almost disjoint family on $\omega_1$. 
We shall develop the ladder systems in a similar way as Lusin sequences. 

\begin{proposition}
\label{cor-4.13}
Let $\mL = \setof{c_\delta}{\delta \in L}$ be a ladder system, where $L$ denotes the set of all infinite countable limit ordinals.
Then there is a ladder system $\mL^\star$  such that 
\begin{itemize}
\item[{\rm(1)}]
$\mL^\star$ has a structure of continuous join-semilattice. 
\item[{\rm(2)}]
There are a subset $B$ of $\omega_1$ and a bijection $\fnn{ h }{ \omega_1 }{ B }$ such that $\mL^\star = \setof{ h[A] }{ A \in \mL }$. 
\end{itemize}
\end{proposition}

\begin{proof}
We set $B_\delta = \Fin \delta$ and let $B = \bigcup_{\delta\in L}B_\delta = \Fin {\omega_1}$.
Let $C_\delta = \Fin{c_\delta}$.
By definition, $\mL^\star = \setof{ C_\delta }{\delta \in L}$. 
We must show that $\mL^\star$ is a ladder system.
Define inductively a well-ordering on $B$, observing the following rule: 
\begin{itemize}
\item[]
Given $\alpha, \beta \in L$ such that $\beta$ is the successor of $\alpha$, the set $B_\beta \setminus B_\alpha$ has order type $\omega$ and $B_\alpha$ is an initial segment of $B_\beta$. \end{itemize}
This is clearly possible, because $B_\beta \setminus B_\alpha$ is infinite and countable whenever $\alpha < \beta$.
Finally, the ordering on $B$ is isomorphic to $\omega_1$ and each $C_\delta$ has order type $\omega$, because $C_\delta \cap B_\alpha$ is finite whenever $\alpha < \delta$.
Thus, $\mL^\star \eqdf \sett{C_\delta}{\delta \in L}$ is a ladder system.
Now, by the construction,  $\mL^\star$ satisfies (1) and (2). 
\end{proof}

Note that the above proof can be easily adapted to more general ladder systems, over a stationary subset $S$ of $\omega_1$. 
The family $\mL^\star$ appearing in Proposition~\ref{cor-4.13},  is an almost disjoint family on $\omega_1$ and we may assume that $\omega_1 = \bigcup \mL^\star$. Therefore:

\begin{corollary}
\label{cor-4.14}
There exists a ladder system $\mL^\star$ such that $K_{\mL^\star}$ is a Mr\'owka space with a continuous join operation. 
\qed
\end{corollary}

\section{Complements on Hyperspaces and on Skula spaces}
\label{section-5}

Given a Priestley space, we complete the relationship between its hyperspace and its Vietoris hyperspace, and we analyze 
the relationship between Skula spaces and topological semilattices. 

\subsection{Priestley hyperspaces versus Vietoris hyperspaces}
\label{section-5.1}

Let $\pair{X}{\leq^X}$ and $\pair{Y}{\leq^Y}$ be two Priestley spaces and let $\fnn{f}{X}{Y}$ be a continuous and increasing map.
We consider the maps:

\begin{itemize}
\item[{\rm}]
$\fnn{ \eta^X }{ X }{H(X) }$ where $\eta^X(x) \eqdf \dd x$ and

\item[{\rm }]
$\fnn{ \eta^Y }{ Y }{H(Y) }$ where $\eta^Y(y) = \dd y$.

\end{itemize} 
Since $f$ and $\eta^Y$ are increasing and continuous,  $\etaf \eqdf \eta^Y \, \circ \, f$ is increasing and continuous.
So by Proposition~\ref{prop-2.4}, there exists a unique  continuous join-semilattice homomorphism
$\fnn { H(f) }{ H(X) }{ H(Y) }$ such that
$H(f) \circ \eta^X = \hat{\eta}$
where $H(X)$ and $H(Y)$ are endowed with the Priestley  topology $\calTT^X$ and $\calTT^Y$ respectively.
So the following diagram
\begin{equation}
\label{semilattice}
\xymatrix{X\ar_f[d]\ar_{\hat\eta}[rrd]\ar^{\eta^X}[rr]&&H(X)\ar@{-->}^{H(f)}[d]\\
Y\ar_{\eta^Y}[rr]&&H(Y)
}
\tag{$\star$}
\end{equation}
is commutative, and thus $H(f) \circ \eta^X = \hat{\eta} = \eta^Y \circ f$.

To a Priestley space $X$, 
we associate the same space $X'$ with the equality relation. 
So $X'$ is also Priestley and $H(X')$, denoted by $\exp(X)$, is the Vietoris hyperspace. 
Since the inclusion map $\fnn{ \imath }{ X' }{ X }$ is increasing and onto, $\imath$ induces an onto  continuous semilattice homomorphism
$\fnn{ \jmath  \, }{ H(X') }{ H(X) }$
satisfying $\jmath \circ \eta^{X'} = \hat{\eta} = \eta^{X }\circ \imath$.
Note that

\begin{itemize}
\item[{\rm{}}]
$\jmath $ is onto and thus $H(X)$ is a continuous image of the compact space $\exp(X)$.
\end{itemize}
On the other hand,
considered as sets, we have, by the definition:
$$
H(X) \, \subseteq \, \exp(X) \, .
$$ 
We denote by $\calTT$ the topology on $H(X)$, 
by $\calTT'$ the topology on $\exp(X)$, and 
by $\calTT^i \eqdf \calTT' {\restriction} H(X)$ the induced topology $\calTT'$ of $\exp(X)$ on $H(X)$.

\begin{proposition}
\label{prop-5.1}
With the above notation, $\calTT \subseteq \calTT^i$
and the following properties are equivalent:

\begin{itemize}
\item[{\rm(i)}]
$H(X)$ is a closed subset of $\exp(X)$.

\item[{\rm(ii)}]
$\calTT = \calTT^i$.

\end{itemize}
\end{proposition}

\begin{proof} 
The inclusion map $\fnn{ {\rm Id} }{ \pair{ H(X) }{ \calTT^i } }{ \pair{ H(X) }{ \calTT } }$  is continuous.
Indeed let $U^+$ be a clopen neighborhood of $U$ in 
$\pair{ H(X) }{ \calTT }$.
So $U^+ \in \calTT$ where $U$ is a clopen initial subset of $X$.   Since $U \in \exp(X)$ we have   
$\setof{ K \in \exp(X) }{ K \subseteq U }   \cap H(X) 
=  {\rm Id}^{-1}[U^+] = U^+ \in \calTT^i$.

(i)$\Rightarrow$(ii)  Suppose that $H(X)$ is closed in $\exp(X)$. 
Since ${\rm Id}$ is continuous, by the compactness of $\pair{ H(X) }{ \calTT^i }$, we have 
$\calTT' {\restriction} H(X) \eqdf \calTT^i = \calTT$.

(ii)$\Rightarrow$(i) 
Suppose $\calTT^i = \calTT$. 
Since ${\rm Id}$ is continuous, by the compactness of $\pair{H(X)}{\calTT}$, the set $H(X)$ is closed in $\pair{ \exp(X) }{ \calTT' }$. 
\end{proof} 

The following simple example shows that $H(X)$ is not necessarily closed in $\exp(X)$.

\begin{example} 
\label{prop-5.2} 
Consider $X = \Nat \cup \{\infty\}$, the one-point compactification of the discrete space $\Nat$.
Then $X$ is a Priestley canonically Skula space when endowed with the ordering $n < \infty$ for every $n \in \Nat$. So, $\Nat$ is an infinite antichain.
Note that $H(X) = \{X\} \cup [\Nat]^{<
\omega}$ 
and $\exp(X)$ consists of all finite sets and all sets containing $\infty$. On the other hand, $H(X)$ is dense in $\exp(X)$ in the Vietoris topology, because every nonempty basic open set of the form $U^+\cap V_1^- \cap \dots \cap V_{k}^-$ contains a finite subset $F$ of $\Nat$, just by picking an element $n_i \in U \cap V_i \cap \Nat$ and setting $F = \{n_1, \dots, n_k\}$. In fact, $[\Nat]^{<\omega}$ is dense in $\exp(X)$.

Note also that $X$ is homeomorphic to the ordinal $\omega+1$ which is again a canonically Skula Priestley space and $H(\omega+1)$ is closed in $\exp(\omega+1)$.
\end{example}

\subsection{Skula spaces and compact semilattices} 
\label{section-5.2}

In the rest of this section, we show that for a compact scattered space $X$, the following properties are independent.
\begin{itemize}
\item[{\rm(1)}]
$X$ has a continuous semilattice operation. 
\item[{\rm(2)}]
$X$ is a Skula space.
\end{itemize}

\begin{proposition}[Banakh and all~\cite{BGR}]
\label{prop-5.3}
Let $\mA$ be maximal almost disjont family on $\omega$.
Then the Mr\'owka space $K_{\mA}$ is not homeomorphic to any topological semilattice.

Moreover $K_{\mA}$ is a separable canonical Skula space. 
\end{proposition}

\begin{proof}
The space $K_{\mA}$ is unitary and $\height(K_{\mA}) = 2 = \rank(K_{\mA})$.
The choice of $\mA$ as a maximal almost disjoint family 
(that is $\mA$ is not contained in a strictly larger 
almost disjoint family on $\omega$) guarantees that $K_{\mA}$
contains no sequence of isolated point that tend to $\infty$.
Then by Theorem 3 of~\cite{BGR} the space $K_{\mA}$ cannot be homeomorphic to a topological semilattice.
\end{proof}

Recall that $(\omega_1 + 1)^2$ is the space of the form $\FS(P)$ where 
$P \eqdf \omega_1 \sqcup \omega_1$ is the disjoint union of two copies of $\omega_1$.  

In the next result, the ``non Skula'' part was proved in terms of Boolean algebras by Bonnet and Rubin: Theorem 4.1 of \cite{BR2}. 
For completeness we show this result using a shorter topological proof.

\begin{proposition} 
\label{prop-5.4} 
Let $X$ be the quotient space of \,$Y \eqdf (\omega_1+1)^2$ 
by the closed ``lower triangle'' 
$\triangle \eqdf \setof{ \pair{\beta}{\gamma} \in \FS(P) }{ \beta \geq \gamma }$.
That is, $X$ is the quotient space $Y {/} {\sim}$ where $\sim$ is the equivalence relation: $x \sim y$ if $x, y \in \triangle$ or $x=y$. 
Then $Y$ is canonically Skula, but $X$ is not Skula and $X$ has a continuous semilatice operation. 
\end{proposition}

\begin{proof}[Boolean sketch] 
The Boolean algebra $B$ of clopen subsets of $Y \eqdf (\omega_1 + 1)^2$ 
is generated by the set 
$G \eqdf \setof{ (\alpha,\beta] \times (\gamma, \delta] }
{\alpha, \beta, \gamma ,  \delta  \leq\omega_1}$.  
Let $B^*$  be the Boolean subalgebra of $B$ generated by the set 
$G^* \eqdf  \{ (\alpha,\beta] \times (\gamma, \delta] \in G : \beta \leq \gamma \}$.
Then $B \eqdf \Clop(Y)$ and $B^* \eqdf \Clop(X)$ are the algebras appearing in Theorem 4.1 of \cite{BR2}: $B$ is canonically well-generated and $B^*$ is not a well-generated subalgebra of $B$. 
In others words, $Y$ has a canonical clopen selector, $X$ is a topological quotient of $Y$ but $X$ has no clopen selector.
\end{proof}

\begin{proof}[Topological proof] 
Since $Y$ is compact and $\triangle$ is closed, $X$ is Hausdorff and compact. 
We denote by $\fnn{q}{Y}{X}$ the quotient map.

For a contradiction assume that $X$ has a clopen selector $\mU = \setof{ U_x }{ x \in X }$. 
For each $y \in Y$ we set $V_y = q^{-1}[U_{q(y)}]$. 
Note that $V_y$ is a clopen neighborhood of $y$ in $X$. 
For simplicity, $V_y = q^{-1}[U_{q(y)}]$ is also denoted by $V_x$ where $x \eqdf q(y) \in X$.

Since $\triangle \in X \eqdf Y/ {\sim}$ and $q(t) = \triangle$ for $t \in \triangle \subseteq Y$, the set $V_\triangle \eqdf q^{-1}[U_{t}]$ is clopen in $Y$
and $V_\triangle$  contains the triangle $\triangle$. 
So for every limit ordinal $\lambda \leq \omega_1$
the set $V_\triangle$ is a neighborhood of $\pair{\lambda}{\lambda}$ in $Y$ and we can find an ordinal $f(\lambda)<\lambda$ such that
$[f(\lambda),\lambda]^2  \subseteq V_\triangle$.
By Fodor Theorem, there are a stationary set $S\subseteq\w_1$ and $\gamma \in \omega_1$ such that 
$f \restriction S = \gamma$. 
We may assume that $\gamma = \min(S)$.
Hence 
\begin{equation}
\label{Eqjediii}
[\gamma,\omega_1)^2 
= \mbox{$\bigcup$}_{\lambda\in S}[\gamma,\lambda]^2  \subseteq V_\triangle \text{ and } [\gamma,\omega_1]^2  
= \cl_Y([\gamma,\omega_1)^2 ) 
\subseteq \cl_Y(V_\triangle) = V_\triangle
\tag{$*$}
\end{equation} 
where $\cl_Y(.)$ denote the topological closure operation (in $Y$). 

For every $\alpha < \omega_1$,  
since $V_{\pair{\alpha}{\omega_1}} \eqdf q^{-1}[U_{q\pair{\alpha}{\omega_1}}]$ is a clopen neighborhood of $\pair{\alpha}{\omega_1}$ in $X$, 
we can find a countable ordinal $g(\alpha)\ge \alpha$
such that $\pair{\alpha}{g(\alpha)} \in V_{\pair{\alpha}{\omega_1}}$. 
Take any point $\alpha_0\in S$ with $\alpha_0\ge\gamma$
and by induction for every $n\in\w$
choose an ordinal $\alpha_{n+1}\in S$ such that
$\alpha_{n+1} > \max(\setof{\alpha_k}{k \leq n}
\cup \setof{ g(\alpha_k)}{k\leq n})$ and choose 
$g(\alpha_{n+1}) < \omega_1$ such that 
$\pair{\alpha_{n+1}}{g(\alpha_{n+1})} 
\in V_{\pair{\alpha_{n+1}}{\omega_1}}$. 
Let $\alpha_\omega = \sup_{n\in\omega} \alpha_n 
= \lim_{n\in\omega} \alpha_n$. 

We claim that  $\pair{\alpha_\omega}{\alpha_\omega} \in V_{\pair{\alpha_\omega}{\omega_1}}$.  
Since the set $V_{\pair{\alpha_\omega}{\omega_1}}$ is closed,
it suffices to check that each clopen neighborhood
$W$ of $\pair{\alpha_\omega}{\alpha_\omega}$ in $X$ meets the set $V_{\pair{\alpha_\omega}{\omega_1}}$. 
From the fact that  
the sequence $\seqnn{\alpha_n}{n\in\omega}$ converges to $\alpha_\omega$, 
$\alpha_\omega = \sup_n g(\alpha_n) 
= \lim_n g(\alpha_n)$, and that 
$W$ and $V_{ \pair{\alpha_\omega }{ \omega_1 } } $, 
there is $m\in\omega$ such that
$[\alpha_m,\alpha_\omega]^2\subseteq W$, 
$\pair{ \alpha_m }{ \omega_1 } 
\in V_{ \pair{\alpha_\omega }{ \omega_1 } } $ and,
by the choice of any $g(\alpha_m)$, 
$\pair{ \alpha_m }{ g(\alpha_m) } 
\in V_{ \pair{ \alpha_m }{ \omega_1 } }$. 
Since $\pair{ \alpha_m }{ \omega_1 } 
\in V_{ \pair{\alpha_\omega }{ \omega_1 } } $ and since $\mU$ is a clopen selector for $X$, 
$V_{\pair{ \alpha_m }{\omega_1 } } 
\subseteq V_{ \pair{\alpha_\omega }{ \omega_1 } } $ 
and thus 
$\pair{ \alpha_m }{ g(\alpha_m } \in V_{ \pair{\alpha_\omega }{ \omega_1 } } $.
On the other hand,
$\pair{\alpha_n}{g(\alpha_n)} \in [\alpha_n,\alpha_\w]^2
\subseteq W$.
Thus $W \cap U_{\la\alpha_\w,\w_1\ra}$
is nonempty and hence 
$\pair{\alpha_\omega}{\alpha_\omega} \in V_{\pair{\alpha_\omega}{\omega_1}}$.

Now since $\pair{\alpha_\omega}{\alpha_\omega}\in V_{\pair{\alpha_\omega}{\omega_1}}$, 
by the definition of $Y/\triangle \eqdf Y/{\sim}$, we have  
$q(\triangle) = q(\pair{\alpha_\omega}{\alpha_\omega}) 
\in U_{\pair{\alpha_\omega}{\omega_1}}$. 
The fact that $\mU$ is a clopen selector implies that 
$U_\triangle \subseteq U_{\pair{\alpha_\omega}{\omega_1}}$. 
On the other hand, 
$\pair{\alpha_\omega}{\omega_1}
\in [\gamma,\omega_1]^2 \subseteq V_\triangle$
and thus 
$q(\pair{\alpha_\omega}{\omega_1}) \in U_{q(\triangle)}$. 
Again since $\mU$ is a clopen selector, $U_{\pair{\alpha_\omega}{\omega_1} } \subseteq U_{q(\triangle)}$. 
Therefore $U_{q(\pair{\alpha_\omega}{\omega_1}) } = U_{q(\triangle)}$ and thus $q(\pair{\alpha_\omega}{\omega_1}) = q(\triangle) \eqdf \triangle$.
But $q(\pair{ \alpha_\omega }{ \omega_1 }) \not\in \triangle$ 
because $\alpha_\omega  \neq  \omega_1$. 
This contradiction shows that $X \eqdf Y/{\sim}$ is not a Skula space.

\smallskip
Next the continuous join operation
$ * : Y\times Y \rightarrow Y$  defined by 
$$(\beta,\gamma) * (\beta',\gamma')
=  (\min\{\beta,\beta'\} , \max\{\gamma,\gamma'\} )$$ induces a continuous semilattice operation $\join$ on $X := Y/{\sim}$, defined by 
$(u/{\sim}) \join (v/{\sim}) := (u*v)/{\sim}$ 
for any $u,v \in Y$ because the singletons and $\triangle$ are closed: 
see also \cite[Theorem 1.54]{CHK1}. 
\end{proof}

\section{Final remarks and open questions}
\label{section-6} 

Recall that any countable scattered compact space is homeomorphic to a countable successor ordinal.
In Proposition~\ref{prop-5.4}, we have seen that there is a canonical Skula space with a non Skula quotient space. 
By ``duality''  we ask the following question. 

\begin{question} 
\label{question-6.1}
Is there an uncountable compact space such that every closed subspace is canonically Skula?
\end{question}

In \S\ref{section-4.1}, we have seen that for a poset $P$, 
the space $\pair{ \FS(P) }{ \supseteq }$ of all final subsets of $P$ and the space $\pair{ \IS(P) }{ \subseteq }$ of all initial subsets of $P$ endowed with the pointwise topology  $\calTT_p$ are  order-isomorphic and homeomorphic.  
Therefore {\em we identify the Priestley spaces $\FS(P)$ and $\IS(P)$.}

In Proposition~\ref{prop-6.3}, we have seen that if $P$ is a well-quasi ordering (well-founded and any set of pairwise incomparable elements is finite), then $\FS(P)$ is a Skula space. 
From this result, M. Pouzet asks for the following question.

\begin{question}[M. Pouzet]
\label{question-6.4}
Let $P$ be a well-quasi ordering. 
Is $\FS(P)$ canonically Skula?
\end{question}

\begin{question}
\label{question-6.5}
Let $P$ be a narrow order-scattered poset (and thus $\FS(P)$ is Skula).
Is $\FS(P)$ canonically Skula? 
\end{question}

In \cite[Theorem 2.1]{BR2}, Bonnet and Rubin proved that every quotient space of  $(\omega_1 + 1) \times (\omega + 1)$ is canonically Skula, and in Proposition~\ref{prop-5.4} we have seen that $(\omega_1 + 1)^2$ has a non-Skula quotient  space.
These facts, in a ``dual'' way, ask for the following question.
\begin{question}
\label{question-6.6}
\begin{itemize}
\item[{\rm(1)}]
Is every closed subset of $(\omega_1 + 1) \times (\omega + 1)$ canonically Skula? 
\item[{\rm(2)}]
Is every closed subset of $(\omega_1 + 1)^2$ canonically Skula?
\end{itemize}
\end{question}

A partial ordering $\pair{P}{\leq}$ has \dfn{finite width}, if for some $n \in \omega$,  $P$ is the union of $n$ chains. 
Note that by Dilworth Theorem, a poset $P$ has finite width whenever there is $n \in \omega$ such that every antichain of $P$ has cardinality~$\leq n$.
Questions \ref{question-6.4}--\ref{question-6.6} ask also for the following.

\begin{question}
\label{question-6.7}
Let $P$ be a well-founded poset of finite width. 
Is $\FS(P)$ canonically Skula?
\end{question}

So we ask for similar cases.

\begin{question}
\label{question-6.8}
\begin{itemize} 
\item[{\rm(1)}]
Let $P$ be a well-founded poset of finite width. 
Is every closed subset of $\FS(P)$ canonically Skula?
\item[{\rm(2)}]
More generally, let $P$ be a narrow and order-scattered poset.
Is every closed subset of $\FS(P)$ canonically Skula?
\end{itemize}
\end{question}

In view of Example~\ref{prop-5.2}, we ask the following question.

\begin{question} 
\label{question-6.9}
Characterize the non trivial Priestley spaces $X$ such that $H(X)$ is closed in $\exp(X)$.
\end{question}

\end{document}